\documentclass[square]{imsart}

\RequirePackage[OT1]{fontenc}
\RequirePackage{natbib}
\RequirePackage[colorlinks,citecolor=blue,urlcolor=blue]{hyperref}

\makeatletter \@addtoreset{equation}{section} \makeatother
\usepackage{graphicx}
\usepackage{amsmath}
\usepackage{amssymb}
\usepackage{amsfonts}
\usepackage{amsthm}
\usepackage{thmtools}
\usepackage{caption}
\usepackage{subcaption}

\usepackage{cleveref}
\usepackage[noend]{algpseudocode}
\usepackage{algorithm}
\usepackage[leftcaption]{sidecap}

\newcommand{\EE}{{\mathbb E}}

\newcommand{\RR}{{\mathbb  R}}

\DeclareMathOperator{\cov}{Cov}

\DeclareMathOperator{\leb}{Leb}

\newtheorem{thm}{Theorem}[section]
\newtheorem{prop}{Proposition}[section]

\newtheorem{rem}{Remark}[section]

\newtheorem{cor}{Corollary}[section]

\advance \topmargin by -\headheight
\advance \topmargin by -\headsep
\advance \topmargin .2in
\begin{document}

\begin{frontmatter}
\title{Efron's monotonicity property for measures on $\RR^2$}  
\runtitle{Efron's theorem} 

\begin{aug}
\author{Adrien Saumard \thanksref{t1}\ead[label=e1]{asaumard@gmail.com}}
\address{CREST, Ensai, Universit{\'e} Bretagne Loire
\\ \printead{e1}}
\thankstext{t1}{Supported by NI-AID grant 2R01 AI29168-04, and by a PIMS postdoctoral fellowship}
\author{Jon A. Wellner\thanksref{t2}\ead[label=e2]{jaw@stat.washington.edu}}
\thankstext{t2}{Supported in part by NSF Grant DMS-1566514, NI-AID grant 2R01 AI291968-04} 

\address{Department of Statistics, University
of Washington, Seattle, WA  98195-4322,\\ \printead{e2}}

\runauthor{Saumard \& Wellner}
\end{aug}
%\begin{aug}
%\author{\fnms{Evan} \snm{Greene}\thanksref{t1}\ead[label=e1]{egreene@uw.edu}}
%\and
%\author{\fnms{Jon A.} \snm{Wellner}\thanksref{t1}\ead[label=e1]{jaw@stat.washington.edu}}
%\ead[label=u1,url]{http://www.stat.washington.edu/jaw/}

%\thankstext{t}{Supported in part by a grant from the Simons Foundation (\#246211)}
%\thankstext{t1}{Supported in part by NSF Grant DMS-1104832} 
%\thankstext{t1}{Supported in part by NSF Grants DMS-1104832 and DMS 1566514, and NI-AID grant 2R01 AI291968-04} 
%\runauthor{Wellner}

%\affiliation{University of Idaho\thanksmark{m1}}
%\affiliation{University of Washington\thanksmark{m2}}

%\address{Department of Statistics \\University of Washington\\Seattle, WA 98195-4322\\
%\printead{e1}}

%\address{Department of Statistics, Box 354322\\University of Washington\\Seattle, WA  98195-4322\\
%\printead{e2}\\
%\printead{u1}
%\end{aug}

\begin{abstract}
First we prove some kernel representations for the covariance of two functions taken on the same random variable
and deduce kernel representations for some functionals of a continuous one-dimensional measure.
Then we apply these formulas to extend Efron's monotonicity property, given in \cite{MR0171335}
and valid for independent log-concave measures, to the case of general measures on $\mathbb{R}^2$. 
The new formulas are also used to derive some further quantitative estimates in Efron's
monotonicity property. 
\end{abstract}



\end{frontmatter}

\tableofcontents

\bigskip

\newpage
%\vspace{1cm}

\section{Introduction : a monotonicity property}
\label{section:intro}

\cite{MR0171335} proved the following proposition:  %more precise statement (see the
%see the remarks after the proof of the main theorem. % in \cite{MR0171335}).

\begin{prop}
\label{prop_equi_dim_2}Let $\left( X,Y\right) $ be a pair of real valued
random variables. Then the following two statements are equivalent:

\begin{description}
\item[(i)] For any $\Psi:\mathbb{R}^{2}\rightarrow \mathbb{R}$, a function
which is nondecreasing in each argument, the conditional expectation
\begin{equation}
I(s)=\mathbb{E}\left[\Psi \left( X,Y\right) \left\vert X+Y=s\right. \right]
\label{def_I}
\end{equation}
is nondecreasing in $s$.

\item[(ii)] For any $\left( x,y\right) \in \mathbb{R}^{2}$, the conditional survival functions%
\begin{equation}
S_X(x;s)=\mathbb{P}\left[ X>x\left\vert X+Y=s\right. \right] \ \ \text{ and }\ \ 
S_Y(y;s)=\mathbb{P}\left[ Y>y\left\vert X+Y=s\right. \right]
\label{def_S}
\end{equation}%
are nondecreasing in $s$.
\end{description}
\end{prop}

\begin{proof}
\textbf{(i)} implies \textbf{(ii)} is given by taking $\Psi (X,Y)=\mathbf{1}_{\left\{ X>x\right\} }$ 
and then by using the symmetry in $X$ and $Y$. 
To prove that \textbf{(ii)} implie \textbf{(i)}, let $F_s^{-1}$ and $G_s^{-1} $ be the conditional 
quantile functions of $X$ and $Y$ given $X+Y =s$;  that is, for $0 < u < 1$
\begin{eqnarray*}
x_{u,s} & \equiv&  F_s^{-1} (u) \equiv \inf \{ x : \ F_s (x) \ge u \},   \\
y_{u,x} & \equiv&  G_s^{-1} (u) \equiv \inf \{ y : \ G_s (y) \ge u \},
\end{eqnarray*}
where $F_s (x) = P(X\le x | X+Y = s)$  and $G_s (y) \equiv P( Y \le y | X+Y =s)$.  
Then, by (ii), for $t < s$,
\begin{eqnarray*}
u & \le & P( X \le F_s^{-1} (u) | X+Y = s) \\
& = & 1 - P(X > F_s^{-1} (u) | X+Y = s ) \\
& \le & 1 - P(X > F_s^{-1} (u) | X+Y = t ) \\
& = & P( X \le F_s^{-1} (u) | X+ Y =t ),
\end{eqnarray*} 
and hence $x_{u,t} = F_t^{-1} (u)\le F_s^{-1}  (u) = x_{u,s}$.  
By symmetry $y_{u,s}$ is also nondecreasing in $s$.  
Thus 
\begin{equation}
\mathbb{E}\left[ \Psi \left( X,Y\right) \left\vert X+Y=s\right. \right]
=\int_{u \in (0,1)} \Psi \left( x_{u ,s},y_{u,s}\right) d u \label{quantile_transf}
\end{equation}%
%where $x_{u ,s}$ is the quantile of order $$ of the conditional
%law of $X$ given $X+Y=s$ and $y_{\alpha ,s}=s-x_{\alpha ,s}$ is the quantile
%of order $\alpha $ of the conditional law of $Y$ given $X+Y=s$. Furthermore,
%we have by \textbf{(ii), }for any $t\leq s$,%
%\begin{eqnarray*}
%\alpha &\leq &\mathbb{P}\left[ X\leq x_{\alpha ,s}\left\vert X+Y=s\right. %
%\right] \\
%&=&1-\mathbb{P}\left[ X>x_{\alpha ,s}\left\vert X+Y=s\right. \right] \\
%&\leq &1-\mathbb{P}\left[ X>x_{\alpha ,s}\left\vert X+Y=t\right. \right] \\
%&=&\mathbb{P}\left[ X\leq x_{\alpha ,s}\left\vert X+Y=t\right. \right] \text{
%.}
%\end{eqnarray*}%
%We deduce that $x_{\alpha ,t}\leq x_{\alpha ,s}$, which means that $%
%x_{\alpha ,s}$ is nondecreasing in $s$. By symmetry, $y_{\alpha ,s}$ is also
%nondecreasing in $s$. 
Then \textbf{(i)} follows from (\ref{quantile_transf})\textbf{.}
\end{proof}

In this paper, condition \textbf{(i)} of Proposition \ref{prop_equi_dim_2}\
is referred to as Efron's ``monotonicity property". \cite{MR0171335} used
Proposition \ref{prop_equi_dim_2} to prove the monotonicity property for
independent log-concave variables $X$ and $Y$. In this paper, we extend the
validity of Efron's monotonicity property to more general pairs $\left(X,Y\right) $ 
on the plane, see Section \ref{section:general_monoton_prop}.  
Our main result, Theorem~\ref{Theorem_rep_cond_quantiles-ver2}, 
provides a condition on the joint density $h$ 
of $(X,Y)$, in terms of the second derivatives 
of $\varphi \equiv (- \log h)$ which imply \textbf{(ii)} of 
Proposition~\ref{prop_equi_dim_2}.
In particular, in Section \ref{ssection_gener_Efron_examples} we exhibit
examples of random pairs satisfying the monotonicity property that are
neither log-concave nor mutually independent. 
We also recover by different techniques Efron's monotonicity for independent 
log-concave variables in Section 3.2. 
Then we obtain quantitative lower-bounds for the derivative of Efron's 
$I$ function in Section 5. 

Our proofs rely on several key covariance identities which are stated in 
Section~\ref{section:rep_form_cov}.  
These identities, originating in \cite{MR0004426}  %Hoeffding} 
(see also \cite{MR1307621} for a translation of the German original), build on 
more recent results in the log-concave case due to \cite{Menz-Otto:2013}.

We conclude the paper in Section~\ref{section:KernelRepresentProofs} %ImproperIntegrals} 
by providing complete proofs of the key 
covariance identities stated in Section~\ref{section:rep_form_cov}.
%kernel representations obtained in Section 2 to the case of improper 
%integrals in Section 6 and compare it with our $L_p$ theory in Section 7.

\begin{rem}\label{remark_1}
It is easily seen, through standard approximation arguments, that point 
\textbf{(ii)} of Proposition \ref{prop_equi_dim_2} is equivalent to
nondecreasingness in $s$ of the functions 
\begin{equation}
\mathbb{E}\left[ \varphi \left( X\right) \left\vert X+Y=s\right. \right] \ \ 
\text{ and }\ \ \mathbb{E}\left[ \varphi \left( Y\right) \left\vert
X+Y=s\right. \right]  \label{One-variable_cond_means}
\end{equation}
for every nondecreasing function $\varphi $. This implies that in point 
\textbf{(i), }one can take without loss of generality functions $\Psi $ to
depend only on one variable. 
A simple proof of the monotonicity of
functionals given in (\ref{One-variable_cond_means}) for independent
log-concave variables $X$ and $Y$ is established in \cite{SauWel2014} using
symmetrization arguments.
\end{rem}

Efron's monotonicity property appears naturally in the theory of log-concave
measures, see \cite{SauWel2014}. Indeed, it has been used by \cite{MR2327839}
and \cite{MR3030616} to prove preservation of ultra log-concavity under
convolution (for discrete random variables), and by \cite{Wellner-2013} to
give a proof that log-concavity and strong log-concavity are preserved by
convolution in the one-dimensional continuous setting. These proofs operate
at the level of scores or relative scores (first derivative of the convex
potentials of the log-concave measures). Without reliance on derivatives,
the classical proof of preservation of log-concavity under convolution
consists of a direct application of Pr\'{e}kopa's theorem, \cite{MR0315079}.
A proof of preservation of log-concavity under convolution can also be
derived via the Brascamp-Lieb inequality (\cite{MR0450480}), that operates
at the second derivative level of the convex potentials and that is the
local form of the Brunn-Minkowski inequality.

Efron's monotonicity property can also be viewed as a monotonicity property
for the collection of conditional laws with respect to the stochastic order
(Theorem 6.B.9. in 
\cite{ShakedShanthi:07}, see also 
\cite{Shanthi:87}, 
\cite{shanthi:87b}, 
\cite{RinottSam:91}, 
\cite{Dubhashi:08}, 
\cite{ZhuangYaoHu:10}). 

Efron's monotonicity property has been applied in the context of 
negative dependence theory (\cite{KumarPro:83}, 
\cite{BlockSavitsShaked:85}, 
\cite{BolandHollanderJoag:96}, 
\cite{HuHu:99}, \cite{Pemantle:00}),
in combinatorial probability (\cite{Fill:88}, \cite{Liggett:00}, \cite{MR2327839}, \cite{GoldMartinSpan}, 
\cite{GrossMansTuckWang:15}), 
in queueing theory (\cite{shanthikumar1986effect}, 
\cite{ShanthiYao:87}, 
\cite{Masuda:95}, 
\cite{PestienRam:02}, 
\cite{DaduSek:04}),
 in Economic theory (\cite{ederer2010feedback}, 
 \cite{wang2012capacity}, 
 \cite{DenuitDhaene:12}), 
  in
the theory of statistical testing (\cite{Berk:78}, 
\cite{CohenSack:87}, 
\cite{CohenSack:90}, 
\cite{BenjaminiHeller:08}, 
\cite{heller2016post}), 
as well as other statistical estimation problems (\cite{stefansk1992monotone}, 
\cite{HwangStef:94}, 
\cite{Ma:99}).

Hence any extension of Efron's monotonicity property may have several applications in statistical theory - and also beyond. 
The questions and issues described in \cite{HwangStef:94} %are simple enough to 
provide an interesting example of the statistical relevance of the results that we obtain below. 
Let us briefly recall the setting of their paper.

\cite{HwangStef:94} study the preservation of monotonicity of regression functions under measurement errors. 
Let $(T,X,U)$ be a triple of random variables where $T$ is a response variable,
$U$ is an (unobserved) covariate, $X = U+Z$ is the covariate $U$ with additive ``measurement error'' $Z$.  
%It amounts to take a triplet $(X,T,U)$ of random variables and 
Hwang and Stefanski discuss  monotonicity of 
$\EE[T|X=x]$ under the assumption that $\EE[T|U=u]$ is monotone. 
Preservation of monotonicity is analyzed relative to the behavior of the measurement error 
$Z:=X-U$.  
% that is interpreted as a measurement error. 
Then the relationship between the ``true" regression function and the regression 
function with ``measurement error'' is important for 
modeling purposes (see \cite{MR857147}, \cite{MR1087101}, \cite{MR1091848}, \cite{MR1064421} and \cite{MR2243417}).

Using Efron's monotonicity property, Hwang and Stefanski show that monotonicity of the regression function 
is preserved when a log-concave error $Z$ in measurement is made independently 
of a log-concave covariate $U$.
Preservation of monotonicity of a regression function will be 
further discussed below in light of our results.

\section{Covariance Identities}
\label{section:rep_form_cov}

Our goal is 
%We want 
to prove the monotonicity property with the greatest generality in
terms of the law of the pair of random variables involved. By Proposition %
\ref{prop_equi_dim_2} above, it suffices to focus on the monotonicity of the
conditional survival functions in  (\ref{def_S}) of \textbf{(ii)}. 
%What are the
%pairs of random variables $\left( X,Y\right) $ for which the functions%
%\begin{equation*}
%S_{X}\left( x,s\right) =\mathbb{P}\left[ X>x\left\vert X+Y=s\right. \right] 
%\text{ and }S_{Y}\left( y,s\right) =\mathbb{P}\left[ Y>y\left\vert
%X+Y=s\right. \right]
%\end{equation*}%
%are increasing in $s$?
To do this in Section~\ref{section:general_monoton_prop} 
we will use several helpful identities for covariances which are summarized below.
Proofs of the new identities in our list, along with examples and counterexamples, 
will be given in Section~\ref{section:KernelRepresentProofs}.
 
It is worth noting that covariance identities have an interest by themselves since they provide powerful tools to derive deviation and concentration inequalities (se for instance \cite{MR1836739}, \cite{MR2060306}, \cite{MR1920106} and also \cite{MR1849347} Section 5.5) or functional inequalities (\cite{Saumard-Wellner:17}). From this point of view, the use of covariance identities to prove extensions of Efron's monotonicity property may be seen as a new connection of covariance identities with functional inequalities.

\begin{prop}
\label{prop_Hoeffding-Shorack}
Suppose that $(X,Y)$ have joint distribution function $H$ on $\RR^2$ with marginal 
distribution functions $F$ and $G$.  Suppose that $a, b $ are non-decreasing functions from 
$\RR$ to $\RR$ with $Var(a(X)) < \infty$ and $Var(b(Y)) < \infty$.  Then
\begin{eqnarray}
Cov[a(X) , b(Y)] = %{\int \! \int}_{\RR^2} 
\int\!\!\!\int_{\mathbb{R}^{2}} \left \{ H(x,y) - F(x) G(y) \right \} d a(x) d b(y) .
\label{Hoeffding-Shorack}
\end{eqnarray}
\end{prop}

The identity (\ref{Hoeffding-Shorack}) can be found in \cite{MR1762415}, 
section 7.4, formula page 117, but it has its origins in 
\cite{MR0004426}  %Hoeffding} 
(see also \cite{MR1307621} for a translation of the German original)
This identity has several useful corollaries.  
We begin 
% and close relatives which we now proceed to enumerate, beginning 
with the original inequality due to \cite{MR0004426}, by taking $a$ and $b$ to be identity functions.
\medskip

\begin{cor}
\label{cor:Hoeffding}
(Hoeffding).  When $a(x) =x$ and $b(y)=y$ for all $x,y \in \RR$, 
\begin{eqnarray*}
Cov[X,Y] =  \int\!\!\!\int_{\mathbb{R}^{2}}  \left \{ H(x,y) - F(x) G(y) \right \} d x d y.
\end{eqnarray*}
\end{cor}

\begin{cor}
\label{cor:Shorack-Conseq} 
(a)  When $Y=X$ almost surely so that $G=F$ and $H(x,y) = F(x \wedge y)$, 
and $a, b$ are nondecreasing and left-continuous,
\begin{eqnarray}
Cov[a(X) , b(X)] 
& = &  \int\!\!\!\int_{\mathbb{R}^{2}} \left \{ F(x \wedge y) - F(x) F(y) \right \} d a(x) d b(y) \nonumber \\
& = &  \int\!\!\!\int_{\mathbb{R}^{2}}  K_{\mu} (x,y) d a(x) d b(y) 
\label{cor:A-B-Kernel}
\end{eqnarray}
where the non-negative and symmetric kernel $K_{\mu }$ on $\mathbb{R}^2$ is defined by   
\begin{equation}
K_{\mu }\left(x,y\right) =F\left( x\wedge y\right) -F\left( x\right) F\left( y\right), 
\qquad  \mbox{for all} \ \ \left( x,y\right) \in \mathbb{R}^{2} .
\label{def_kernel}
\end{equation}
and where 
$F\left( x\right) =F_{\mu }\left( x\right) =\mu \left( \left( -\infty ,x\right] \right) $ is the  distribution function associated with the 
probability measure $\mu $ on $(\mathbb{R}, {\cal B})$. \\
(b)  Moreover, (\ref{cor:A-B-Kernel}) continues to hold if $a = a_1 -a_2$, $b=b_1 -b_2$ where $a_j \in L_p (F)$ and $b_j \in L_q (F)$ for $j=1,2$
with $p^{-1} + q^{-1} =1$.
\end{cor}

Now we specialize Corollary~\ref{cor:Shorack-Conseq} slightly by taking $a$ to be an indicator function.
\smallskip

\begin{cor}
\label{cor:IndicatorAFcn}
Suppose that $b = b_1 - b_2$ where $b_1, b_2$ are left-continuous and non-decreasing with either 
$b_j \in L_2 (F)$ for $j=1,2$ or $b_j \in L_1 (F)$ for $j=1,2$, and let $z \in \RR$. 
Then, with $a(x) = 1_{(-\infty,z]} (x)$,
\begin{eqnarray}
F(z) \int_{\RR} b dF - \int_{(-\infty,z]} b dF 
 =  - \ Cov[ 1_{[X \le z]} , b(X) ] = \int_{\RR} K_{\mu} (z,y) db(y) ,
\label{Indicator-MRLCovIdent-1}
\end{eqnarray}
and
\begin{eqnarray}
- (1-F(z)) \int_{\RR} b dF + \int_{(z,\infty)} b dF 
 =  Cov[ 1_{[X > z]} , b(X) ] = \int_{\RR} K_{\mu} (z,y) db(y) .
\label{Indicator-MRLCovIdent-2}
\end{eqnarray}
Furthermore, if $b \in L_1 (F)$ is absolutely continuous, then
\begin{eqnarray}
F(z) \int_{\RR} b dF - \int_{(-\infty,z]} b dF 
 =  - \ Cov[ 1_{[X \le z]} , b(X) ] = \int_{\RR} K_{\mu} (z,y) b'(y) dy ,
\label{Indicator-MRLCovIdent-1-AbsC}
\end{eqnarray}
and
\begin{eqnarray}
- (1-F(z)) \int_{\RR} b dF + \int_{(z,\infty)} b dF 
 =  Cov[ 1_{[X > z]} , b(X) ] = \int_{\RR} K_{\mu} (z,y) b'(y) dy .
\label{Indicator-MRLCovIdent-2-AbsC}
\end{eqnarray}
\end{cor}

\begin{rem}
Note that the quantities appearing on the left sides in (\ref{Indicator-MRLCovIdent-1}) %CovIdentity3-A}) 
and (\ref{Indicator-MRLCovIdent-2}) have 
interpretations in terms of {\sl mean residual life} or {\sl reversed mean residual life}
functions:  in particular, the left side of (\ref{Indicator-MRLCovIdent-2}) can be written as 
\begin{eqnarray*}
 \left ( E \{ h(X) | X>z \} - E\{ h(X) \} \right ) (1-F(z)),
\end{eqnarray*}
while the left side of (\ref{Indicator-MRLCovIdent-1}) can be written as 
\begin{eqnarray*}
- \left ( E \{ h(X) | X \le z \} - E\{ h(X) \} \right ) F(z) .
\end{eqnarray*}
\end{rem}
\smallskip

Our next corollary, a further corollary of Corollary~\ref{cor:Shorack-Conseq}, % of Proposition~\ref{prop_Hoeffding-Shorack} 
allows the functions $a$ and $b$ to be differences of left-continuous and non-decreasing 
functions, or absolutely continuous.  
\medskip

\begin{cor}
\label{cor:DifferenceVersion1}  
(Menz and Otto)
If $a= a_1 - a_2$ and $b = b_1 - b_2$ where $a_j  \in L_p (F)$ and $b_j  \in L_q (F)$ for $j=1,2$  with  $p^{-1} + q^{-1} =1$, 
then (\ref{cor:A-B-Kernel}) continues to hold.   Moreover, if $a$ and $b$ are absolutely continuous with $a \in L_p (F)$ 
and $b \in L_q (F)$, then 
\begin{eqnarray}
Cov[a(X) , b(X)] 
& = &  \int\!\!\!\int_{\mathbb{R}^{2}} a'(x)  K_{\mu} (x,y) b'(y) d x d y .
\label{cor:A-B-KernelAbsCont}
\end{eqnarray}
\end{cor}
\medskip

The  covariance identity (\ref{cor:A-B-KernelAbsCont}) appeared in \cite{Menz-Otto:2013} 
(but without explicit assumptions on the functions $a$ and $b$). 
Note that this inequality implies a version of the FKG inequality:  
if $a $ and $b$ are non-decreasing, then $a' (x) \ge 0 $ and $b'(y) \ge 0$ so that the right side 
of (\ref{cor:A-B-KernelAbsCont}) is non-negative, and hence $E\{ a(X) b(Y) \} \ge E\{ a(X) \} E \{ b(Y) \}$.
\medskip

\smallskip

%The following lemmas giving representation formulas for the covariance
%and conditional quantiles in terms of the kernel $K_{\mu }$ will be useful.
Our last set of covariance identities involve taking $b = \varphi^{\prime} $ 
in the case when $F$ has density $f = \exp (- \varphi)$.

\begin{cor}
\label{cor:KernelRecoveryOfDensity}
%\textbf{Corollary 3.} 
Suppose that $F$ has absolutely continuous density $f=\exp(-\varphi )$.\\
%\noindent 
(a) If $\varphi $ has derivative $\varphi ^{\prime }$ which
satisfies $\varphi ^{\prime }=\varphi _{1}^{\prime }-\varphi _{2}^{\prime }$
where $\varphi _{j}^{\prime }\in L_{1}(F)$ for $j=1,2$ and 
$\varphi_{j}^{\prime }$ are left-continuous and non-decreasing, 
% and $\varphi_j^{\prime} $ is absolutely continuous for $j=1,2$.  
then %with the same notation as in Corollary~\ref{cor:ExtendedCovIdenityFour},   
\begin{equation}
\int_{\mathbb{R}}K_{\mu }(x,y)d\varphi ^{\prime }(y)=f(x).
\label{LogConcaveTypeIdentityGeneralForm}
\end{equation}
%\textcolor{cyan}{Adrien: is the condition $\varphi$ absolutely %continuous needed in point (a)(see the remarks in the proof) ?}\\
%\textcolor{magenta}{Jon:  it now seems to me that the right %hypothesis is that $f$ is absolutely continuous.  I have made %that 
%a blanket hypothesis at the beginning of the Corollary.}\\
(b) If $\varphi $ has derivative $\varphi ^{\prime }\in L_{1}\left( F\right) 
$ which is absolutely continuous, then 
\begin{equation}
\int_{\mathbb{R}}K_{\mu }(x,y)\varphi ^{\prime \prime }(y)dy=f(x).
\label{LogConcaveTypeIdentitySecondDerivativeForm}
\end{equation}
(c) In particular, if $f$ is log-concave and absolutely continuous, $f=\exp (-\varphi )$ with $\varphi 
$ convex, then (\ref{LogConcaveTypeIdentityGeneralForm}) holds. \\
(d) If $f$ is log-concave and absolutely continuous, and $\varphi ^{\prime }$ is absolutely continuous,
then (\ref{LogConcaveTypeIdentitySecondDerivativeForm}) holds. 
\end{cor}
\smallskip

%\par\noindent
%\textcolor{magenta}{Jon:   counterexample to (c):
%  Let $f(x) = e^{-x} 1_{[0,\infty)} (x)$.  
%Then $f$ is log-concave: $\varphi (x) = x 1_{[0,\infty)} (x) + \infty \cdot 1_{(-\infty,0)} (x)$ is convex.  
%Note that $\varphi'(x) = 1_{[0,\infty)}(x)$ 
%exists for all $x>0$, but $\varphi $ is not absolutely continuous because of the discontinuity of $\varphi$ at $0$.}
%\textcolor{cyan}{(Adrien: is it legitimate to consider that $\varphi'(x)$ exists for $x<0$ since $\varphi(x)$ is infinite for $x<0$?} \\
%\textcolor{magenta}{ Jon: No clearly not. \\
%I calculate
%\begin{eqnarray*}
%\int_0^{\infty} K_{\mu} (x,y) d \varphi ' (y) 
%& = & \int_0^{\infty} \{ (1- e^{- x\wedge y}) - (1-e^{-x}) (1-e^{-y} ) \} d 1_{[0,\infty)} (y) \\
%& = & \left \{ (1 - e^{-x\wedge 0} ) - (1- e^{-x} )\cdot 0 \right \} = 0 
%\end{eqnarray*}
%for all $x \ge 0$.  Some further hypothesis needs to be added to (c)?!   Here is a proposed revision of (c):\\
%(c) In particular, if $f$ is log-concave, $f=\exp (-\varphi )$ with $\varphi 
%$ convex and absolutely continuous, then (\ref{LogConcaveTypeIdentityGeneralForm}) holds. }\\
%\noindent 
%It is worth 
The condition $\varphi ^{\prime }\in L_{1}\left( F\right) $ in (b) of 
Corollary \ref{cor:KernelRecoveryOfDensity} is not overly restrictive. 
%does not bring severe restriction. 
Indeed, it is equivalent to $f ^{\prime }\in L_{1}\left( \leb \right) $ since $f^{\prime}=-\varphi^{\prime}f$, 
and the latter condition is easily checked.  Also note that $\varphi^{\prime} = - f'/f$ is the ``score for location'' 
in statistics.  
%So, in modelisation contexts, for instance for heavy-tailed distributions related to 
%extreme statistics, density $f$ will systematically be nonincreasing at the 
%neighborhood of $\infty$ and thus $f^{\prime}$ will be of constant sign at the 
%neighborhood of $\infty$. Hence, the condition $f ^{\prime }\in L_{1}\left( \leb \right) $ 
%will be equivalent to $\int f^{\prime}dx<+\infty$ and thus automatically satisfied.
\begin{rem}
Corollary~\ref{cor:KernelRecoveryOfDensity}(d) was given by \cite{Menz-Otto:2013}.
The other parts of Corollary~\ref{cor:KernelRecoveryOfDensity} are apparently new. %Otto and Menz.
\end{rem}

Proofs of Proposition~\ref{prop_Hoeffding-Shorack}  
and Corollaries~\ref{cor:Hoeffding} - \ref{cor:KernelRecoveryOfDensity}  will be given in Section~\ref{section:KernelRepresentProofs}.

\smallskip

\section{The monotonicity property for general measures on $\mathbb{R}^2$} 
\label{section:general_monoton_prop}

\subsection{A general result}

Proposition~\ref{prop_equi_dim_2} shows that monotonicity of $s \mapsto I(s)$ in (\ref{def_I}) is implied by monotonicity 
of the conditional survival functions $S_X( x; s)$ and $S_Y (y;s)$ in (\ref{def_S}). 
The following theorem provides a way of verifying the monotonicity of the conditional 
survival functions $s \mapsto S_X(x;s)$ and $s\mapsto S_Y( y; s)$ in terms of the elements of the Hessian matrix
$\mbox{Hess} (\varphi )$ where $\varphi = -\log h$ is the potential (perhaps non-convex) of the joint density $h$ of $(X,Y)$.  
First some further notation.  
We write 
\begin{eqnarray*}
\mbox{Hess} ( \varphi )(x,y) 
= \left ( \begin{array}{c c} \frac{\partial^2}{\partial x^2} \varphi (x,y) & \frac{\partial^2}{\partial y \partial x} \varphi (x,y)  \\
                                         \frac{\partial^2}{\partial x \partial y} \varphi (x,y) &  \frac{\partial^2}{\partial y^2 } \varphi (x,y) \end{array} \right ) 
    \equiv \left ( \begin{array}{c c}  \partial_{11}^2 \varphi  &    \partial_{21}^2 \varphi  \\  \partial_{12}^2 \varphi  &    \partial_{22}^2 \varphi  
                        \end{array} \right ) (x,y) 
\end{eqnarray*} 
for the Hessian of $\varphi \equiv - \log h$ where we suppose that $h>0$ on some open set $S \subset \RR^2$. 
We denote the conditional densities of $X$ given $X+Y = s$ and $Y$ given $X+Y=s$ by $f_1 (\cdot; s) \equiv f_1$ and
$f_2 (\cdot ; s) \equiv f_2$ respectively, and denote the corresponding conditional measures by 
$\mu_1 \equiv \mu_1 (\cdot ; s)$ and $\mu_2  \equiv \mu_2 (\cdot ; s)$ respectively.  
Thus 
\begin{eqnarray}
f_1 (x; s) & = & \exp (- \varphi_1 (x; s)) = \frac{h(x,s-x)}{\int_{\RR} h(x', s-x') dx'} , \label{CondDensOne}\\
f_2 (y; s) & = & \exp (- \varphi_2 (y; s)) = \frac{h(s-y,y)}{\int_{\RR} h(s-y', y') dy'} .  \label{CondDensTwo}
\end{eqnarray}
Furthermore we write 
\begin{eqnarray*}
\partial_1 \varphi (x,y) \equiv \frac{\partial}{\partial x} \varphi (x,y) \ \ \ \mbox{and} \ \ \ 
\partial_2 \varphi (x,y) \equiv \frac{\partial}{\partial y} \varphi (x,y).
\end{eqnarray*}
We will also need the following domination conditions:  
\medskip

\par\noindent
{\bf D1:}  Fix $s_0 \in \RR$.  Suppose that $x \mapsto \partial_2 \varphi (x,s-x)$ is absolutely continuous 
for $s \in [s_0-\epsilon, s_0 + \epsilon]\equiv S_{\epsilon}$ for some $\epsilon > 0$, 
and there exists a function $\overline{g} \in L_1 (Leb)$ such that 
\begin{eqnarray*}
| \partial_2 \varphi (x, s-x)  \exp ( - \varphi (x, s-x) ) | \le \overline{g}(x)
\end{eqnarray*}
for almost all $x\in \RR \cap S$ and all $s \in S_{\epsilon}$.
\medskip

\par\noindent
{\bf D2:}  Fix $s_0 \in \RR$.  Suppose that $y \mapsto \partial_1 \varphi (s-y,y)$ is absolutely continuous 
for $s \in S_{\epsilon}$ for some $\epsilon > 0$, 
and there exists a function $\overline{h} \in L_1 (Leb)$ such that 
\begin{eqnarray*}
| \partial_1 \varphi (s-y, y)  \exp ( - \varphi (s-y, y) ) | \le \overline{h}(y)
\end{eqnarray*}
for almost all $y\in \RR \cap S$ and all $s \in S_{\epsilon}$.
\medskip

\begin{thm} 
\label{Theorem_rep_cond_quantiles-ver2}
Suppose that D1 holds.  Then with $K_{1,0} \equiv K_{\mu_{1,s_0}}$,
\begin{eqnarray}
\partial_2 S_X (x,s_0) = \int_{\RR} K_{1,0} (x,x') ( \partial_{22}^2 \varphi - \partial_{12}^2 \varphi ) (x', s_0 -x' ) d x' . 
\label{DerivSsubX}
\end{eqnarray}
Suppose that D2 holds.  Then with $K_{2,0} \equiv K_{\mu_{2,s_0}}$,
\begin{eqnarray}
\partial_2 S_Y (y ; s_0)
&  = & \int_{\RR} K_{2,0} ( y', y ) ( \partial_{11}^2 \varphi - \partial_{21}^2 \varphi ) (s_0 -y' ,y') d y' 
             \label{DerivSsubY}\\
& = & \int_{\RR} K_{1,0} (s_0 - y, y') ( \partial_{11}^2 \varphi - \partial_{21}^2 \varphi ) (y', s_0 -y' ) d y' .
\label{DerivSsubYForm2}
\end{eqnarray}
Moreover, if both D1 and D2 hold, then 
\begin{eqnarray}
f_1 (x; s_0 ) & = & \partial_2 S_X (x;s_0) + \partial_2 S_Y (s_0-x; s_0) , \label{CondDens1_SumOfDeriv}\\
f_2 (y; s_0) & = & \partial_2 S_Y (y; s_0 ) + \partial_2 S_X (s_0-y; s_0)  \label{CondDens2_SumOfDeriv}.
\end{eqnarray}
\end{thm}

\medskip

\begin{cor}  \label{cor_potential_conditions}
If  D1 and D2 hold (so the conclusions of Theorem~\ref{Theorem_rep_cond_quantiles-ver2} hold), and 
\begin{eqnarray*}
&& (\partial_{22}^2 \varphi - \partial_{12}^2 \varphi ) (x', s_0 -x' ) \ge 0 \ \ \mbox{for all} \ \ x' \text{ such that }(x', s_0 -x' )\in S \ \ \mbox{and} \\
&&  (\partial_{11}^2 \varphi - \partial_{21}^2 \varphi ) (s_0 -y' ,y') \ge 0 \ \ \mbox{for all} \ \ y' \text{ such that }(s_0 -y' ,y')\in S,   
\end{eqnarray*}
then the conditional survival functions $S_X (\cdot | s)= S_X (\cdot ; s)$ and $S_Y (\cdot | s) = S_Y (\cdot ; s)$ in (ii) of
Proposition~\ref{prop_equi_dim_2}  
 are non-decreasing in $s$, and hence (i) of Proposition~\ref{prop_equi_dim_2}
also holds.  
\end{cor}
\medskip

\begin{proof}
Let $J_1 (s) \equiv \log \left ( \int_{\RR} h(x,s-x) dx \right )$.  By the domination assumption D1, the function 
$J_1 $ is differentiable on $S_{\epsilon}$ with derivative
\begin{eqnarray*}
J_1' (s) = - \int_{\RR}  \partial_2 \varphi (x, s-x) h(x,s-x) dx \bigg /  \int_{\RR} h(x', s-x') dx' .
\end{eqnarray*}
Note that $\varphi_1 (x; s) = \varphi (x,s-x) + J_1 (s) $, and we therefore find that
\begin{eqnarray*}
\partial_2 \varphi_1 (x;s) = \partial_2 \varphi (x,s-x) + J_1^{\prime} (s),
\end{eqnarray*}
and 
\begin{eqnarray}
\partial_{12}^2 \varphi_1 (x; s) = ( \partial_{12}^2 \varphi- \partial_{22}^2\varphi ) ( x,s-x) .
\label{VarPhiOneDerivFormula}
\end{eqnarray}
Multiplying by minus one and integrating this identity with respect to $K_{\mu_1} (x,x')$ and then 
applying covariance identity (\ref{Indicator-MRLCovIdent-1-AbsC}) 
yields
\begin{eqnarray}
\lefteqn{\int_{\RR} ( \partial_{22}^2 \varphi - \partial_{12}^2  \varphi )(x', s-x') K_{\mu_1} (x,x') dx' } \nonumber \\
& = & - \int_{\RR} K_{\mu_1} ( x,x') \partial_{12}^2 \varphi_1 (x';s) dx'   \qquad \ \ \mbox{by (\ref{VarPhiOneDerivFormula})}  \nonumber \\
& = & - \left ( \int_{(-\infty,x]} f_1 (x'; s) dx' \int_{\RR} \partial_2 \varphi_1 (x';s) f_1(x';s) dx'  
                - \int_{(-\infty,x]} \partial_2 \varphi_1 (x'; s) f_1 (x';s) dx' \right ) 
   \label{IntermediateIdentityMuOne}
\end{eqnarray}
where the first term is 
\begin{eqnarray*}
\int_{\RR} \partial_2 \varphi_1 (x'; s_0) f_1 (x'; s_0) dx' 
& = & \int_{\RR} \partial_2 \varphi_1 (x';s_0) \exp (- \varphi_1 (x';s_0)) dx' \\
& = & - \frac{d}{ds} \int_{\RR} \exp (- \varphi_1 (x'; s) )dx' \bigg |_{s = s_0} = 0 ,
\end{eqnarray*}
and where the second term is
\begin{eqnarray*}
\int_{(-\infty, x]} \partial_2 \varphi_1 (x'; s_0) f_1 (x'; s_0) dx' 
& = &   \int_{(-\infty,x]} \partial \varphi_1 (x';s_0) \exp (- \varphi_1 (x';s_0)) dx'  \bigg |_{s = s_0} \\
& = & - \frac{\partial}{\partial s} \left (\mu_2 (-\infty,x] \right )(s_0) = \frac{\partial}{\partial s} S_X (x; s_0) .
\end{eqnarray*}
Combining this with (\ref{IntermediateIdentityMuOne})  evaluted at $s= s_0$ 
yields the claimed identity (\ref{DerivSsubX}).  

The identity (\ref{DerivSsubY}) follows from the same argument used to prove (\ref{DerivSsubX}) by  symmetry.  
To prove (\ref{DerivSsubYForm2}), let $F_{1,s}$ and $F_{2,s}$ denote the conditional distribution functions 
corresponding to the conditional densities $f_1 ( \cdot ; s)$ and $f_2 (\cdot ; s)$.  Then $F_{1,s} (x) = 1-F_{2,s} (s-x)$ 
and hence with $K_{j,s} (x,y) = F_{j,s} (x \wedge y) - F_{j,s} (x) F_{j,s} (y)$, $j=1,2$, it follows that 
\begin{eqnarray}
K_{2,s} (y, y') = K_{1,s} (s-y, s-y')   \ \ \mbox{for all} \ \ y, y' . 
\label{ConditionalKernelIdentity}
\end{eqnarray}
Then, evaluating at $s=s_0$, 
\begin{eqnarray*}
\lefteqn{\int_{\RR} K_{2,0} (y' , y) ( \partial_{11}^2 \varphi - \partial_{21}^2 \varphi )(s_0-y' , y' ) dy' } \\
& = & \int_{\RR} K_{2,0} (y , y') ( \partial_{11}^2 \varphi - \partial_{21}^2 \varphi )(s_0-y' , y' ) dy'  \ \ \ \mbox{since} \ \ K_{2,0} (y,y') = K_{2,0} (y',y) \\
& = &  \int_{\RR} K_{1,0} (s-y ,s- y') ( \partial_{11}^2 \varphi - \partial_{21}^2 \varphi )(s_0-y' , y' ) dy'  \ \ \ \mbox{by} \ \ (\ref{ConditionalKernelIdentity}) \\
& = & \int_{\RR} K_{1,0} (s-y , v) ( \partial_{11}^2 \varphi - \partial_{21}^2 \varphi )(v , s_0 -v ) dv \\
&& \qquad \qquad \ \ \mbox{by the change of variable} \ \ y' = s_0 -v .
\end{eqnarray*}
Thus (\ref{DerivSsubYForm2}) holds.   

To show that (\ref{CondDens1_SumOfDeriv}) holds, note that since
$\varphi_1 (x; s) = \varphi (x,s-x) + J_1 (s)$,
\begin{eqnarray*}
\partial_1 \varphi_1 (x; s) & = & \partial_1 \varphi (x,s-x) - \partial_2 \varphi (x,s-x) ,  \ \ \ \mbox{and} \\
\partial_{11} \varphi_1 (x; s) & = & \partial_{11}^2 \varphi (x,s-x) - \partial_{21}^2 \varphi (x,s-x) + \partial_{12}^2 \varphi_{22} (x,s-x) .
\end{eqnarray*}
But then
\begin{eqnarray*}
f_1 (x; s_0) 
& = & \int_{\RR} K_{1,0} (x, x') \partial_{11}^2 \varphi_1 (x'; s_0) dx' \ \ \ \mbox{by (\ref{LogConcaveTypeIdentitySecondDerivativeForm}), 
          \ Corollary~\ref{cor:KernelRecoveryOfDensity} } \\
& = &  \int_{\RR} K_{1,0} (x, x')  
           \left \{ \partial_{11}^2 \varphi (x',s_0-x') - \partial_{21}^2 \varphi (x',s_0-x') + \partial_{12}^2 \varphi_{22} (x',s_0-x') \right \} d x' \\
& = & \partial_2 S_X (x;s_0) + \partial_2 S_Y (y;s_0)
\end{eqnarray*}
by using (\ref{DerivSsubX}) and (\ref{DerivSsubYForm2}) in the last equality.
\end{proof}
\bigskip

Now we are ready discuss examples (and counter-examples) of joint distributions on 
$\mathbb{R}^{2}$ where \textbf{(ii)} of  Proposition~\ref{prop_equi_dim_2}  is satisfied. 
Identities (\ref{DerivSsubX}) - (\ref{DerivSsubYForm2}) %(\ref{formula_p_X_kernel}) and (\ref{formula_p_Y_kernel}) 
in Theorem \ref{Theorem_rep_cond_quantiles-ver2}
are very useful in this regard.

But we first consider the log-concave case in the light of Theorem \ref{Theorem_rep_cond_quantiles-ver2}.

\subsection{Independent log-concave variables revisited}
\label{section_indep_log_concav}

The following theorem is due to \cite{MR0171335}. We give a different proof
than Efron's, based on formulas (\ref{DerivSsubX}) %(\ref{formula_p_X_kernel}) and 
and (\ref{DerivSsubY}) %(\ref{formula_p_Y_kernel}) 
of Theorem \ref{Theorem_rep_cond_quantiles-ver2} above.

\begin{thm}[\protect\cite{MR0171335}]
\label{Theorem_Efron_68}
The monotonicity property is satisfied for any pair
of independent log-concave random variables.
\end{thm}

\begin{proof}
Let $\left( X,Y\right) $ be a pair of independent log-concave random
variables with density $h$ on $\mathbb{R}^{2}$ with respect to Lebesgue measure. 
Then
\begin{equation*}
h\left( x,y\right) 
=g_{X}\left( x\right) g_{Y}\left( y\right) =\exp \left(
-\left( \varphi _{X}\left( x\right) +\varphi _{Y}\left( y\right) \right)
\right) ,\text{ \ \ }\left( x,y\right) \in \mathbb{R}^{2},
\end{equation*}%
where $g_{X}$ and $g_{Y}$ are the densities ($\varphi _{X}$ and $\varphi
_{Y} $ are the convex potentials) of $X$ and $Y$ respectively. Indeed, a
log-concave random variable on $\mathbb{R}$ automatically has a density with
respect to the Lebesgue measure (see for instance \cite{SauWel2014}). 
For now, let us also assume that $h>0$ on $\mathbb{R}^{2}$. \\
Denote also 
$\varphi \left( x,y\right) =\varphi_{X}\left( x\right) +\varphi _{Y}\left( y\right) $,
$\left( x,y\right) \in \mathbb{R}^{2}$. 
Let us first assume that $\varphi _{X}$ and 
$\varphi _{Y}$ are $C^{2}$ and that $g_{X}^{\prime },g_{Y}^{\prime }\in L_{\infty }$.
%\textcolor{magenta}{Jon:  what are $g_X$ and $g_Y$ here?  Should these be $\varphi_X$ and $\varphi_Y$?}\\
Define measures $\mu _{1}$ by $d\mu _{1}\left( x\right) =f_1\left( x;s\right)dx$ where 
\begin{equation*}
f_1 \left( x; s\right) 
=\exp \left( -\varphi _{1}\left( x;s\right) \right) 
= \frac{h\left( x,s-x\right) }{\int_{\mathbb{R}}h\left( x^{\prime},s-x^{\prime }\right) dx^{\prime }}\text{.}
\end{equation*}
Then, it follows that 
\begin{equation*}
\varphi _{1}\left( x;s\right) 
=\varphi _{X}\left( x\right) +\varphi_{Y}\left( s-x\right) + \log  \left( \int_{\mathbb{R}}
   \exp \left( -\left( \varphi _{X}\left( x^{\prime }\right) 
   +\varphi _{Y}\left( s-x^{\prime}\right) \right) \right) dx^{\prime }\right) \text{ .}
\end{equation*}%
Using the assumption that that 
$g_{Y}^{\prime }=-\varphi _{Y}^{\prime }\exp \left( -\varphi _{Y}\right) $ 
is uniformly bounded, it is easy to see that 
$\partial _{2} \varphi _{1}$ exists and that 
\begin{equation*}
\partial _{2}\varphi _{1}\left( x; s\right) =\varphi _{Y}^{\prime }\left(
s-x\right) +\frac{\int_{\mathbb{R}}\varphi _{Y}^{\prime }\left( s-x^{\prime
}\right) \exp \left( -\left( \varphi _{X}\left( x^{\prime }\right) 
+\varphi_{Y}\left( s-x^{\prime }\right) \right) \right) dx^{\prime }}
           {\int_{\mathbb{R}}h\left( x^{\prime },s-x^{\prime }\right) dx^{\prime }}~\text{.}
\end{equation*}
Again using  $g_{Y}^{\prime }\in L_{\infty }$, simple calculations show
that there exists a function $\overline{g}\in L_{1}\left( \leb \right) $ 
such that for almost all $x\in \mathbb{R}$, 
\begin{equation*}
\left\vert \partial _{2}\varphi _{1}\left( x;s\right) \exp \left( -\varphi_{1}\left( x;s\right) \right) \right\vert \leq  \overline{g} \left( x\right) \text{.}
\end{equation*}%
Furthermore, $\partial _{2}\varphi _{1}\left( \cdot ;s\right) $ is
absolutely continuous (even $C^{1}$), so by formula (\ref{DerivSsubX}), %formula_p_X_kernel}), 
it follows that
\begin{eqnarray*}
\left( \partial _{2}S_{X}\right) \left( x;s\right) 
&=&\int_{\mathbb{R}}K_{\mu _{1}}\left( x,u\right) 
        \left( \partial _{22}^{2}\varphi -\partial_{12}^{2}\varphi \right) \left( u,s-u\right) du \\
&=&\int_{\mathbb{R}}K_{\mu _{1}}\left( x,u\right) \varphi _{Y}^{\prime \prime }\left( s-u\right) du\text{ .}
\end{eqnarray*}
Since $\varphi_Y^{\prime \prime} \ge 0$ by log-concavity of $g_Y$,
%Hence, 
it follows that 
$\left( \partial _{2}S_{X}\right) \left( x;s\right) \geq 0$. 
Note that the argument shows that even if $g_X$ is is not log-concave, log-concavity of $g_Y$
implies monotonicity of $s \mapsto S_X(x;s)$.

By symmetry between $X$ and $Y$, we also have $\left( \partial _{2}S_{Y}\right)
\left( y;s\right) \geq 0$ and we conclude from Proposition \ref{prop_equi_dim_2} 
that the monotonicity property is satisfied for $\left(X,Y\right) $.

To conclude, we have to prove that we can reduce the situation from general
convex potentials to potentials $\varphi $ that are finite on $\RR$ (this implies that $h>0$ on $\RR$), that are $C^{2}$ and that satisfy 
$\left\Vert \varphi ^{\prime }\exp \left( -\varphi \right) \right\Vert
_{\infty }<+\infty $. This is done by convolution with Gaussian random
variables, whose variance tends to zero (see for instance Proposition 5.5 in 
\cite{SauWel2014}). In particular, we see that any pair of independent
log-concave random variables $\left( X,Y\right) $ there exists a sequence of
log-concave random variables $\left( X_{n},Y_{n}\right) $, with $X_{n}$
independent of $Y_{n}$, such that the densities $g_{X_{n}}$ and $g_{Y_{n}}$
of $X_{n}$ and $Y_{n}$ are $C^{\infty }$ and converge respectively to $g_{X}$
and $g_{Y}$ in $L_{\infty }$. Hence, for any $\left( x,y,s\right) \in 
\mathbb{R}^{3}$,%
\begin{equation*}
S_{X_{n}}(x; s) \rightarrow S_{X} (x; s) 
\text{ \ \ and \ \ }
S_{Y_{n}}(y; s) \rightarrow S_{Y}(y; s) \text{ ,} \ \mbox{as} \ {n\rightarrow \infty }
\end{equation*}
which gives the monotonicity in $s$ of $S_{X}(x; s) $ and $S_{Y}(y; s) $.
\end{proof}

The monotonicity extends to more than two independent log-concave variables.

\begin{thm}[\protect\cite{MR0171335}]
\label{theorem_Efron}Let $m$ be greater than one. Then the monotonicity
property is satisfied for any $m$-tuple of independent log-concave variables.
\end{thm}

\begin{proof}
We proceed as in \cite{MR0171335} by induction on $m$. Let 
$\left(X_{1},\ldots ,X_{m}\right) $ be an $m-$tuple of log-concave variables, let
$S=\sum_{i=1}^{m}X_{i}$ be their sum, and set 
\begin{equation*}
\Lambda \left( t,u\right) 
=\mathbb{E}\left[ \Psi \left( X_{1},\ldots,X_{m}\right) \left\vert \sum_{i=1}^{m-1}X_{i}=t\text{ },\text{ }%
X_{m}=u\right. \right]  .
\end{equation*}
Then 
\begin{equation*}
\mathbb{E}\left[ \Psi \left( X_{1},\ldots ,X_{m}\right) \left\vert S=s\right. \right] 
=\mathbb{E}\left[ \Lambda \left( T,X_{m}\right)
\left\vert T+X_{m}=s\right. \right] \text{ ,}
\end{equation*}
where $T=\sum_{i=1}^{m-1}X_{i}$. The variable $T$ has a log-concave density
(by preservation of log-concavity by convolution). Hence, by the induction
hypothesis at rank $2$, it suffices to prove that $\Lambda $ is
coordinatewise non-decreasing. $\Lambda \left( t,u\right) $ is
non-decreasing in $t$ by the induction hypothesis at rank $m-1$. Also 
$\Lambda \left( t,u\right ) $ is non-decreasing in $u$ since $\Psi $ is
non-decreasing in its last argument. This concludes the proof.
\end{proof}

\subsection{Examples\label{ssection_gener_Efron_examples}}

\subsubsection{Bivariate Gaussian}

This special case, in which the joint density is log-concave but not independent, 
 is simple but instructive. Suppose that $(X,Y) \sim N_2
(0, \Sigma)$ where 
\begin{equation*}
\Sigma = \left ( 
\begin{array}{cc}
\sigma^2 & \rho \sigma \tau \\ 
\rho \sigma \tau & \tau^2
\end{array}
\right ) .
\end{equation*}
Then 
\begin{equation*}
\varphi (x,y) = - \log \phi_{\Sigma} (x,y) = \frac{1}{2 (1-\rho^2)} \left ( 
\frac{x^2}{\sigma^2} - 2\rho \frac{x}{\sigma} \frac{y}{\tau} + \frac{y^2}{\tau^2} \right ) + \mbox{constant} ,
\end{equation*}
so that 
\begin{eqnarray*}
&& \frac{\partial}{\partial x} \varphi (x,y) = \frac{1}{\sigma^2 (1-\rho^2)}
\left ( x - \frac{\rho \sigma}{\tau} y \right ) , \\
&& \frac{\partial}{\partial y} \varphi (x,y) = \frac{1}{\tau^2 (1-\rho^2)}
\left ( y - \frac{\rho \tau}{\sigma} x \right ) ,
\end{eqnarray*}
and 
\begin{eqnarray*}
&& \partial_{11}^2 \varphi (x,y) = \frac{1}{\sigma^2 (1-\rho^2)}, \ \ \ \
\partial_{22}^2  \varphi (x,y) = \frac{1}{\tau^2 (1-\rho^2)}, \\
&& \partial_{12}^2 \varphi (x,y) = \partial_{21}^2 \varphi (x,y) = - \frac{\rho}{\sigma \tau (1-\rho^2)} .
\end{eqnarray*}
Thus we have 
\begin{eqnarray*}
&& \partial_{11}^2\varphi(x,y)- \partial_{21}^2 \varphi (x,y) = \frac{1}{\sigma
(1-\rho^2)} \left ( \frac{1}{\sigma} + \frac{\rho}{\tau} \right ) , \\
&& \partial_{22}^2\varphi(x,y)- \partial_{21}^2 \varphi (x,y) = \frac{1}{\tau
(1-\rho^2)} \left ( \frac{1}{\tau} + \frac{\rho}{\sigma} \right ) ,
\end{eqnarray*}
where the right hand sides of the last two displays are nonnegative if and
only if $-\rho \le (\tau/\sigma) \wedge (\sigma/ \tau)$ or, equivalently, if
and only if $\rho \ge - \{ (\tau/\sigma) \wedge (\sigma/ \tau)\}$.  It follows 
from Corollary~\ref{cor_potential_conditions} that if  
$\rho \ge - \{ (\tau/\sigma) \wedge (\sigma/ \tau)\}$, then 
$\mathbb{E}\left[\Psi \left( X,Y\right) \left\vert X+Y=z\right. \right]$
is a monotone function of $z$ for any function $\Psi$ which is monotone in each coordinate,
for example $\Psi_1 (x,y) \equiv 1 \{ x \ge 0, y \ge 0 \}$ or $\Psi_2 (x,y) \equiv ax + b y$ with $a, b \ge 0$.

In fact, in this example we have $(X | X+Y = z) \sim N(\mu z, A^2)$  and, by symmetry, 
$(Y | X+Y = z) \sim N(\nu z, A^2)$ where 
\begin{eqnarray*}
&& \frac{1}{A^2} = \frac{1}{1-\rho^2} \left \{ \frac{1}{\sigma^2} + \frac{2
\rho}{\sigma \tau} + \frac{1}{\tau^2} \right \} , \\
&& \mu = \frac{A^2}{1-\rho^2} \left ( 1 + \rho\frac{\tau}{\sigma} \right ) \frac{1}{\tau^2} , \\
&& \nu = \frac{A^2}{1-\rho^2} \left ( 1 + \rho\frac{\sigma}{\tau} \right ) \frac{1}{\sigma^2} = 1 - \mu .
\end{eqnarray*}
(Note that when $\sigma = \tau =1$ and $\rho = 0$ this yields $(X | X+Y = z)
\sim N( z/2, 1/2)$.)

Now we check the claimed monotonicity for the conditional expectations in the case of $\Psi_1 $ and $\Psi_2$.
%What function(s) $\Psi $ to consider here? One interesting choice is $\Psi(x,y)=1\{x\geq 0,y\geq 0\}$. 
For $\Psi_1$, with $\Phi (z) \equiv P(N(0,1)\leq z)$ on the right side, 
\begin{eqnarray*}
E\{\Psi_1 (X,Y)|X+Y=z\} 
&=&P(X\geq 0,z-X\geq 0|X+Y=z)=P_{z}(0\leq X\leq z) \\
&=&\{\Phi ((z-\mu z)/A)-\Phi (-\mu z/A)\}1\{z\geq 0\} \\
& = & \{ \Phi (\nu z /A) - \Phi (- \mu z/A)\} 1 \{ z \geq 0 \} .
\end{eqnarray*} 
This is a monotone function of $z$ if %and only if 
$\rho \geq -\{\sigma /\tau \wedge \tau /\sigma \}$. 
%Another example of interest is $\Psi (x,y) = ax +by $ with $a,b \ge 0$. 
%Then $\Psi$ is monotone in each coordinate and 
For $\Psi_2$ we have 
\begin{eqnarray*}
E\{\Psi_2 (X,Y)|X+Y=z \}  
& = & a E(X | X+Y = z)  + b E(Y | X+Y = z) = a \mu z + b \nu z \\
& = & (a \mu + b \nu) z =  \frac{A^2}{1-\rho^2} \left \{ \frac{a}{\tau^2} \left ( 1 + \frac{\rho \tau}{\sigma} \right )  
                + \frac{b}{\sigma^2} \left ( 1 + \frac{\rho \sigma}{\tau} \right ) \right \} z .
%= \left \{ 1 + (a-b) \frac{A^2}{\tau^2 (1-\rho^2)} (1 + \rho \frac{\tau}{\sigma} ) \right \} z .
\end{eqnarray*}
This is monotone increasing for any $a, b \ge 0$ if and only if $ \rho \ge - \{ (\tau/\sigma) \wedge (\sigma/ \tau)\}$,
just as we concluded above via Corollary~\ref{cor_potential_conditions}.
%\begin{eqnarray*}
%a-b \left \{ 
%\begin{array}{ll}
%\ge \frac{-\tau^2 (1-\rho^2)}{(1 + \rho \tau/\sigma)}, & \mbox{when} \ \
%\rho \ge - \sigma/\tau , \\ 
%\le \frac{-\tau^2 (1-\rho^2)}{(1 + \rho \tau/\sigma)}, & \mbox{when} \ \
%\rho < - \sigma/ \tau .%
%\end{array}
%\right .
%\end{eqnarray*}
%\subsection{Other variables}

\subsubsection{Morgenstern copula}  

(Not log-concave and not independent)
Suppose that $(X,Y)$ has density $c_{\theta}$ on $[0,1]^2$ where 
\begin{eqnarray*}
c_{\theta} (x,y) = 1+ \theta (1-2x)(1-2y), \ \ \ (x,y) \in [0,1]^2
\end{eqnarray*}
for $| \theta | \le 1$. Then straightforward calculation yields 
\begin{eqnarray*}
&& \partial_{11}^2  \varphi (x,y) - \partial_{21}^2 \varphi (x,y) = \frac{4
\theta (1 + \theta (1-2y)^2)}{[1 + \theta (1-2x)(1-2y)]^2}, \\
&& \partial_{22}^2 \varphi (x,y) - \partial_{12}^2 \varphi (x,y) = \frac{4
\theta (1 + \theta (1-2x)^2)}{[1 + \theta (1-2x)(1-2y)]^2}
\end{eqnarray*}
and the right sides in the last display are both non-negative if and only if 
$\theta\ge 0$. Hence for $\Psi$ coordinatewise monotone, $E \{\Psi (X,Y) |X+Y = z \}$ 
is monotone in $z$ if and only if $\theta \ge 0$.

\subsubsection{Frank copula}

(Not log-concave and not independent)  Suppose that $(X,Y)$ has distribution function $C_{\theta}$ on $[0,1]^2$
where 
\begin{eqnarray*}
C_{\theta} (x,y) = \left \{ 
\begin{array}{ll}
\log \left \{ 1 - \frac{(1-\theta^x) (1-\theta^y)}{(1-\theta)} \right \} /
\log \theta , & \theta \not= 1, \ (x,y) \in [0,1]^2 \\ 
xy, & \theta = 1, \ (x,y) \in [0,1]^2 .%
\end{array}
\right .
\end{eqnarray*}
for $0 < \theta < \infty $. Then straightforward calculation yields 
\begin{eqnarray*}
&& \partial_{11}^2 \varphi (x,y) - \partial_{21}^2 \varphi (x,y) = - \frac{2
\theta^x (\theta - 2\theta^y + \theta^{2y} ) (\log \theta )^2}{(\theta -
\theta^x - \theta^y + \theta^{x+y} )^2}, \\
&& \partial_{22}^2 \varphi (x,y) - \partial_{12}^2 \varphi (x,y) = - \frac{2
\theta^y (\theta - 2\theta^x + \theta^{2x} ) (\log \theta )^2}{(\theta -
\theta^x - \theta^y + \theta^{x+y} )^2}
\end{eqnarray*}
and the right sides in the last display are both non-negative if and only if 
$\theta \in (0,1]$.  
Hence for $\Psi $ coordinatewise monotone, $E \{ \Psi(X,Y) | X+Y = z \}$ 
is monotone in $z$ if and only if $0 < \theta \le 1$.

Note that $\theta=1$ corresponds to $(X,Y)$ being independent uniform
$(0,1)$ random variables, and we know that the conditional expectation is monotone 
by Efron's theorem in this case.  

\subsubsection{Clayton-Oakes copula}

(Not log-concave and not independent) 
Suppose that $(X,Y)$ has distribution function $C_{\theta }$ on $(0,1]^{2}$
where 
\begin{equation*}
C_{\theta }(x,y)=\left\{ x^{-\theta }+y^{-\theta }-1\right\} ^{-1/\theta },\
\ \ (x,y)\in (0,1]^{2}
\end{equation*}%
for $0<\theta <\infty $. Then straightforward calculation yields 
\begin{eqnarray*}
\lefteqn{\partial _{11}^2 \varphi (x,y)-\partial _{21}^2 \varphi (x,y)} \\
&=&\frac{\theta (1+2\theta )x^{1+\theta }y^{\theta }+\theta y^{1+2\theta
}-(1-\theta -2\theta ^{2})x^{\theta }y^{1+\theta }(1-y^{\theta })-(1-\theta
)x^{2\theta }y(1-y^{\theta })^{2}}{x^{2}y(y^{\theta }+x^{\theta
}(1-y^{\theta }))^{2}},
\end{eqnarray*}
\begin{eqnarray*}
\lefteqn{\partial _{22}^2 \varphi (x,y)-\partial _{12}^2 \varphi (x,y)} \\
&=&\frac{\theta (1+2\theta )y^{1+\theta }x^{\theta }+\theta x^{1+2\theta}
       -(1-\theta -2\theta ^{2})y^{\theta }x^{1+\theta }(1-x^{\theta })-(1+\theta )y^{2\theta }x(1-x^{\theta })^{2}}{y^{2}x(x^{\theta }
       +y^{\theta}(1-x^{\theta }))^{2}}.
\end{eqnarray*}
and the right sides in the last displays are both non-negative if $\theta \in
(1/2,1)$. Hence for $\Psi $ coordinatewise monotone, $E\{ \Psi (X,Y)|X+Y=z\}$
is monotone in $z$ if $1/2\leq \theta \leq 1$.

\subsection{Monotonicity preservation under measurement error}
\label{subsec:MonUnderMsmtError}

Now we discuss the statistical application described at the end of Section \ref{section:intro} above
 in light of Theorem~\ref{Theorem_rep_cond_quantiles-ver2}. 
Indeed, we are now able to %substantially 
extend the results of \cite{HwangStef:94} related to independent log-concave errors.
Briefly recall the framework:  
we are given a triple $(T,X,U)$ of random variables, and consider the monotonicity of 
$\EE[T|X=x]$ under the assumption that $\EE[T|U=u]$ is monotone. 
The variable $Z:=X-U$ is interpreted as a measurement error and is essential in the analysis. 
We assume in the sequel that the triple $(T,X, U)$ has a density 
$f_{T,U,X}$ with respect to Lebesgue measure on $\RR^{3}$, 
and that $T$ and $X$ are conditionally independent given $U$. %By consequence, we have
Thus
\begin{equation*}
f_{T,U,X}(t,u,x)=f_{T|U}(t|u)f_{X|U}(x|u)f_U(u) .
\end{equation*}
This yields: %which in particular gives the following formula,
\begin{equation}\label{forumla_meas_error}
\EE[T|X=x]=\int\EE[T|U=u]f_{U|X}(u|x)du.
\end{equation}%
Setting $\Psi(u)=\EE[T|U=u]$, formula (\ref{forumla_meas_error}) can be rewritten as follows:  
\begin{equation}\label{forumla_meas_error_2}
\EE[T|X=x]=\EE[\Psi(U)|U+Z=x].
\end{equation}%
If $\Psi$ is nondecreasing, we see by Remark \ref{remark_1} that monotonicity of 
$\EE[T|X=x]$ is thus ensured as soon the conditional quantiles 
\begin{equation*}
S_U(u,x)=\mathbb{P}\left[ U>u\left\vert U+Z=x\right. \right]
\end{equation*}%
are nondecreasing in $x$ for any $u\in\RR$. Now, setting 
$\varphi=-\log(f_{U,X})$ and using Corollary \ref{cor_potential_conditions}, 
we get that monotonicity of $\EE[T|X=x]$ is ensured if for every $s_{0}\in\RR$, 
Assumption {\bf D1} is valid for $\varphi$ and for all $x'\in\RR$,
\begin{equation}
(\partial_{22}^2 \varphi - \partial_{12}^2 \varphi ) (x', s_0 -x' ) \ge 0 .
\end{equation}
As seen in the examples above, this condition is valid for independent log-concave variables $U$ and $Z$, 
but also for possibly dependent variables which may or may not be log-concave. 
\section{Quantitative estimates in the monotonicity property}
\label{section:application_to_monotonicity}

In this section, we establish a quantitative strengthening of Efron's monotonicity property.
Recall that we are interested in the function $I$ of $s\in \mathbb{\mathbb{R}}$, given in (\ref{def_I}).
We thus consider a pair $\left( X,Y\right) $ of
random variables with density $h$ on $\mathbb{R}^{2}$ with respect to the
Lebesgue measure. By setting%
\begin{equation*}
S_{X}\left( x,s\right) =\mathbb{P}\left[ X>x\left\vert X+Y=s\right. \right] 
\text{ and }S_{Y}\left( y,s\right) =\mathbb{P}\left[ Y>y\left\vert
X+Y=s\right. \right] \text{ ,}
\end{equation*}%
we have seen in Section \ref{section:intro} that the function $I$ is
non-decreasing if for all $(x,y)\in \mathbb{R}^{2}$, $S_{X}\left( x,s\right) 
$ and $S_{Y}\left( y,s\right) $ are non-decreasing in $s\in \mathbb{R}$.

Note that if $h$ is positive and continuous on $\mathbb{R}^{2}$ then $\int_{%
\mathbb{R}}h\left( s-y^{\prime },y^{\prime }\right) dy^{\prime }>0$ and the
function $f_{2}$ given by (\ref{CondDensTwo})
%\begin{equation}
%f_{1}\left( y;s\right) =\frac{h\left( s-y,y\right) }
%      {\int_{\mathbb{R}}h\left( s-y^{\prime },y^{\prime }\right) dy^{\prime }}  \label{def_m1}
%\end{equation}
is well-defined. In this case, we may write
\begin{equation}
I\left( s\right) =\int_{\mathbb{R}} \Psi \left( s-y,y\right) f_{2}\left( y;s\right) dy\text{ .}  \label{I_integral_form}
\end{equation}
By a change of variable, we may also write
\begin{equation*}
I\left( s\right) =\int_{\mathbb{R}} \Psi \left( x,s-x\right) f_{1}\left( x; s\right) dx\text{ ,}
\end{equation*}
where $f_1$ is given by (\ref{CondDensOne}).
%\begin{equation}
%f_{2}\left( x,s\right) =\frac{h\left( x,s-x\right) }{\int_{\mathbb{R}%
%}h\left( x^{\prime },s-x^{\prime }\right) dx^{\prime }}\text{ .}
%\label{def_m}
%\end{equation}
We define the measure $\mu _{1}$ by $d\mu _{1}\left( x\right) =f_{1}\left(x;s\right) dx$.
\smallskip

\begin{thm}
\label{theorem_Efron_lower_bound}Assume that the statements of Propostion \ref{prop_equi_dim_2}, that is Efron's monotonicity property, hold. Let us take $s_{0}\in \mathbb{R}$ and 
$\varepsilon >0$, and let  $V\left( s_{0}\right) =\left[ s_{0}-\varepsilon,s_{0}+\varepsilon \right] $. 
With the notations above, assume that $h=\exp \left( -\varphi \right) $ is positive and coordinatewise 
differentiable on $\mathbb{R}^{2}$. 
Assume also that $\Psi $ is coordinatewise differentiable
on $\mathbb{R}^{2}$. Furthermore, assume that for any 
$s\in V\left( s_{0}\right) $, $\Psi \left( \cdot ,s-\cdot \right) $ and 
$\left( \partial_{1}\varphi \right) \left( \cdot ,s-\cdot \right) $ are absolutely
continuous. Assume that for all $\left( x,y\right) \in \mathbb{R}^{2}$, the
functions $S_{X}\left( x,s\right) $ and $S_{Y}\left( y,s\right) $ are
non-decreasing in $s\in V\left( s_{0}\right) $.

If there exist four integrable functions on $\mathbb{R}$, 
$A,B,C,D\in L_{1}\left( \leb \right) $ 
and a positive constant $M$ such that, for all 
$\left( s,x,y\right) \in V\left( s_{0}\right) \times \mathbb{R}^{2},$ 
\begin{eqnarray}
\left\vert \Psi \left( s-y,y\right) \right\vert &\leq &M \text{ ,}
\label{line1_quant} \\
\left\vert \partial _{1} \Psi \left( s-y,y\right) h\left( s-y,y\right)
\right\vert &\leq &A\left( y\right) \text{ ,}  \label{line2_quant} \\
\left\vert \partial _{1}\varphi \left( s-y,y\right) h\left( s-y,y\right)
\right\vert &\leq &B\left( y\right) \text{ ,}  \label{line3_quant} \\
\left\vert \partial _{2}\varphi \left( x,s-x\right) \exp \left( -\varphi
\left( x,s-x\right) \right) \right\vert &\leq &C\left( x\right) \wedge
D\left( s-x\right) \text{ ,}  \label{line4_quant}
\end{eqnarray}
then the function $I$ defined in (\ref{def_I}) is differentiable at the point $s_{0}$ and
\begin{equation}
I^{\prime}\left(s_{0}\right)   
 \geq  \left( 1-\sup_{x\in \mathbb{R}}\left\{ 
                          \frac{\partial _{2}S_{Y}\left( s_0-x,s_0\right) }
                                 {\partial _{2}S_{X}\left( x,s_0\right) +\partial _{2}S_{Y}\left( s_0-x,s_0\right) }\right\} 
                         \right) \mathbb{E}\left[ \left( \partial _{1} \Psi \right) \left( X,Y\right) \left\vert X+Y=s_0\right. \right]  . \label{mino_deriv} \end{equation}
\end{thm}

\begin{proof}[Proof of Theorem \protect\ref{theorem_Efron_lower_bound}]
Under the assumptions of Theorem \ref{theorem_Efron_lower_bound}, the
function $I$ in (\ref{CondDensTwo}) and (\ref{I_integral_form}) is well-defined and we have (using notations above),
\begin{equation*}
I\left( s\right) =\int_{\mathbb{R}} \Psi \left( s-y,y\right) f_{2}\left( y; s\right) dy\text{ .}
\end{equation*}%
By differentiating with respect to $s$ (interchanging differentiation and
integral signs is allowed by the assumptions (\ref{line1_quant}), 
(\ref{line2_quant}) and (\ref{line3_quant})), we get 
\begin{equation}
I^{\prime }\left( s_{0}\right) 
=\mathbb{E}\left[ \left( \partial _{1} \Psi
\right) \left( X,Y\right) \left\vert X+Y=s_{0}\right. \right] 
-\cov\left[ \Psi \left( X,Y\right) ,\left( \partial _{1}\varphi \right) \left(
      X,Y\right) \left\vert X+Y=s_{0}\right. \right] \text{ .}
\label{derivative_I}
\end{equation}%
Notice that by Assumption (\ref{line4_quant}), kernel representations hold
for $\partial _{2}S_{X}$ and $S_{Y}$. Now, by Corollary \ref{cor:DifferenceVersion1}, Theorem 
\ref{Theorem_rep_cond_quantiles-ver2} and coordinatewise monotonicity of $\Psi$, we have 
\begin{eqnarray}
\lefteqn{\cov\left[ \Psi \left( X,Y\right) ,\left( \partial _{1}\varphi
\right) \left( X,Y\right) \left\vert X+Y=s_{0}\right. \right] } \\
&=&\label{cov_x_B_L_1} \int \!\!\int \left( \partial _{1} \Psi -\partial _{2} \Psi \right) \left( x,s_{0}-x\right) 
        K_{\mu _{s_{0}}}\left( x,\tilde{x}\right) \left( \partial_{11}^{2}\varphi -\partial _{12}^{2}\varphi \right) \left( \tilde{x},s_{0}-\tilde{x}\right) dxd\tilde{x} \\
&=& \label{cov_x_B_L_2} \int \left( \partial _{1} \Psi -\partial _{2} \Psi \right) \left( x,s_{0}-x\right) \left( \partial _{2}S_{Y}\right) \left(
             s_{0}-x,s_{0}\right) dx \\
&\leq & \label{cov_x_B_L_3} \int \left( \partial _{1} \Psi  \right) \left( x,s_{0}-x\right) \left(
               \partial _{2}S_{Y}\right) \left( s_{0}-x,s_{0}\right) dx \\
&\leq & \label{cov_x_B_L_4} \sup_{x\in \mathbb{R}}\left\{ \frac{\left( \partial _{2}S_{Y}\right)
              \left( s_{0}-x,s_{0}\right) }{f_{1}\left( x,s_{0}\right) }\right\} \int
              \left( \partial _{1} \Psi  \right) \left( u,s_{0}-u\right)
             f_{1}\left( u,s_{0}\right) du  \notag \\
&=& \label{cov_x_B_L_5} \sup_{x\in \mathbb{R}}\left\{ \frac{\left( \partial _{2}S_{Y}\right)
        \left( s_{0}-x,s_{0}\right) }{\left( \partial _{2}S_{X}\right) \left(
         x,s_{0}\right) +\left( \partial _{2}S_{Y}\right) \left( s_{0}-x,s_{0}\right)  }\right\} \mathbb{E}\left[ \left( \partial _{1} \Psi \right) \left(
X,Y\right) \left\vert X+Y=s_{0}\right. \right] \text{ .}  
\end{eqnarray}
Indeed, equality (\ref{cov_x_B_L_1}) comes from identity (\ref{cor:A-B-KernelAbsCont}), then we used identity (\ref{DerivSsubYForm2}) to obtain (\ref{cov_x_B_L_2}). Inequality (\ref{cov_x_B_L_3}) is derived using coordinatewise monotonicity of $\Psi$ together with monotonicity of $S_{Y} \left(y,s_{0}\right)$. Finally, equality (\ref{cov_x_B_L_5}) follows from identity (\ref{CondDens1_SumOfDeriv}). 
\end{proof}
\medskip
Note that by symmetry between $X$ and $Y$, if the right integrability conditions are satisfied, then we could also get
\begin{eqnarray*}
I^{\prime}\left(s_{0}\right)   
& \geq &  \left( 1-\sup_{x\in \mathbb{R}}\left\{ \frac{\partial _{2}S_{X}\left( x,s_0\right) }{\partial _{2}S_{X}\left( x,s_0\right) 
            +\partial _{2}S_{Y}\left( s_0-x,s_0\right) }\right\} \right) 
            \mathbb{E}\left[ \left( \partial _{2} \Psi \right) \left( X,Y\right) \left\vert X+Y=s_0\right. \right]  \notag
\end{eqnarray*}
or, mixing the latter lower bound with the one of Theorem \ref{theorem_Efron_lower_bound},
\begin{eqnarray*}
I^{\prime}\left(s_{0}\right)   
& \geq & \left( 1-\sup_{x\in \mathbb{R}}\left\{ 
                          \frac{\partial _{2}S_{Y}\left( s_0-x,s_0\right) }
                                 {\partial _{2}S_{X}\left( x,s_0\right) +\partial _{2}S_{Y}\left( s_0-x,s\right) }\right\} 
                         \right) \mathbb{E}\left[ \left( \partial _{1} \Psi \right) \left( X,Y\right) \left\vert X+Y=s_0\right. \right]   \\
&& \ \ \vee \left( 1-\sup_{x\in \mathbb{R}}\left\{ \frac{\partial _{2}S_{X}\left( x,s_0\right) }{\partial _{2}S_{X}\left( x,s_0\right) 
            +\partial _{2}S_{Y}\left( s_0-x,s_0\right) }\right\} \right) 
            \mathbb{E}\left[ \left( \partial _{2} \Psi \right) \left( X,Y\right) \left\vert X+Y=s_0\right. \right] .  \notag
\end{eqnarray*}
Let us now return to the statistical application discussed at the end of the introduction and further investigated in 
Subsection~\ref{subsec:MonUnderMsmtError} %application_to_monotonicity} 
above. 
Using the notation of Subsection~\ref{subsec:MonUnderMsmtError}, 
we are now able to give a lower bound on the derivative of the regression function $\EE[T|X=x]$, 
relative to the derivative of $\Psi(u)=\EE[T|U=u]$. 
Indeed, by setting $h=f_{U,X}=\exp(-\varphi)$, 
we find from formula (\ref{forumla_meas_error_2}) and Theorem \ref{theorem_Efron_lower_bound} 
that, under the required integrability conditions we have
\begin{equation*}
\frac{d}{dx}(\EE[T|X=x]) \ge \left(1- \sup_{u\in \mathbb{R}}\left\{ 
                          \frac{\partial _{2}S_{Z}\left( x-u,x\right) }
                                 {\partial _{2}S_{U}\left( u,x\right) +\partial _{2}S_{Z}\left( x-u,x\right) }\right\} 
                         \right) \mathbb{E}\left[  \Psi'  \left( U\right) \left\vert U+Z=x\right. \right] .
\end{equation*}

\section{Proofs and Examples for Section~\protect\ref{section:rep_form_cov}} %of the Covariance Identities}%Some kernel representations for improper integrals}
\label{section:KernelRepresentProofs}

\subsection{Proofs of the Covariance Identities}
\label{subsec:ProofsCovIdent}

We begin by reviewing several identities in \cite{MR1762415}, section 7.4.
Let $(X,Y)$ have a joint distribution function $H$ on $\mathbb{R}^2$ with marginals 
$F$ and $G$ for $X$ and $Y$ respectively.  
Let $F^{-1}$ denote the left-continuous inverse of $F$. Thus if $\xi \sim \ $Uniform$(0,1)
$, $X\equiv F^{-1}(\xi )$ has distribution function $F$. Then we can write 
\begin{equation*}
X=\int_{(0,1)}F^{-1}(t)d1_{[\xi \leq t]}=F^{-1}(\xi ),\ \ \mbox{and}\ \
X=\int_{\mathbb{R}}xd1_{[X\leq x]}.
\end{equation*}%
Similarly if the mean $\mu $ of $X$ exists, then 
\begin{equation*}
\mu =\int_{(0,1)}F^{-1}(t)dt , \ \ \mbox{and} \ \ \nu = \int_{\mathbb{R}}xdF(x).
\end{equation*}%
By taking the differences in these identities we find that 
\begin{eqnarray*}
&&X-\mu =\int_{(0,1)}F^{-1}(t)d(1_{[\xi \leq t]}-t)=-\int_{(0,1)}(1_{[\xi \leq t]}-t)dF^{-1}(t),\ \ \ \mbox{and} \\
&&X-\mu =\int_{\mathbb{R}}xd(1_{[X\leq x]}-F(x)) = -\int_{\mathbb{R}}(1_{[X\leq x]}-F(x))dx,
\end{eqnarray*}%
where the second expressions follow from integration by parts or from
Fubini's theorem. Note that the existence of $\mu $ is used in both of
these proofs. 
Similarly, if $a$ is nondecreasing and left-continuous, with $E|a(X)|<\infty $, then 
\begin{equation*}
a(X)=\int_{\mathbb{R}} a(x)d1_{[X\leq x]},\ \ \ \ Ea(X)=\int_{\mathbb{R}}a(x)dF(x),
\end{equation*}%
and hence 
\begin{equation*}
a(X)-Ea(X)=\int_{\mathbb{R}}a(x)d(1_{[X\leq x]}-F(x)) = -\int_{\mathbb{R}}(1_{[X\leq x]}-F(x))da(x) .
\end{equation*}
A similar development for $b(Y)$ yields
\begin{equation*}
b(Y)-Eb(Y)=\int_{\mathbb{R}}b(y)d(1_{[Y\leq y]}- G(y)) = -\int_{\mathbb{R}}(1_{[Y\leq y]}-G(y))db(y) .
\end{equation*}
where the second expressions follows from integration by parts or from Fubini's theorem together 
with $Var(b(Y)) < \infty$.

Using the identities above together with Fubini's theorem and the assumption 
$Var(a(X))<\infty $, $Var(b(Y)) < \infty$, we obtain the covariance identity (\ref{Hoeffding-Shorack}).
This is just as in \cite{MR1762415}, page 117, formula (14).
\smallskip 

Corollary~\ref{cor:Hoeffding} follows immediately from Proposition~\ref{Hoeffding-Shorack} by taking 
$a$ and $b$ to be the identity functions.   

\begin{proof}[Proof of Corollary~\protect\ref{cor:Shorack-Conseq}]
Part (a) follows immediately upon noting that when $Y=X$, $H(x,y) = F(x\wedge y)$. \\
Part (b):  First note that $Var(a(X))<\infty $ since 
\begin{eqnarray*}
Var(a(X)) &\leq &Ea^{2}(X)=E\{(a_{1}(X)-a_{2}(X))^{2}\} 
%&\leq &E\{2(a_{1}^{2}(X)+a_{2}^{2}(X))\}\leq
\le 2\{Ea_{1}^{2}(X)+Ea_{2}^{2}(X)\}<\infty .
\end{eqnarray*}
Similarly, $Var(b(X)) < \infty$, and  
 $|Cov(a_{1}(X),a_{2}(X))|<\infty $ by the Cauchy-Schwarz inequality.  Then $Cov(a_1 (X), a_2 (X))$ is given
 by polarization: 
\begin{equation*}
Cov(a_{1}(X),a_{2}(X))=\frac{1}{4}\left\{
Var(a_{1}(X)+a_{2}(X))-Var(a_{1}(X)-a_{2}(X))\right\} .
\end{equation*}%
Thus we have, by using the variance identity resulting from (\ref{cor:A-B-Kernel}) with $a=b$, the covariance identity (\ref{cor:A-B-Kernel}), 
and the symmetry of $F(x\wedge y)-F(x)F(y)$ in $x$ and $y$, 
\begin{eqnarray*}
\lefteqn{Var(a(X))} \\
&=&Var(a_{1}(X)-a_{2}(X)) \\
&=&Var(a_{1}(X))-2Cov(a_{1}(X),a_{2}(X))+Var(a_{2}(X)) \\
&=&\int \!\!\!\int \{F(x\wedge y)-F(x)F(y)\} da_{1}(x)da_{1}(y)
        - \int \!\!\!\int \{F(x\wedge y)-F(x)F(y)\}d a_{1}(x)d a_{2}(y) \\
&&\ -\ \int \!\!\!\int \{F(x\wedge y)-F(x)F(y)\}da_{1}(y)da_{2}(x)+\int
        \!\!\!\int \{F(x\wedge y)-F(x)F(y)\}da_{2}(x)da_{2}(y) \\
&=&\int \!\!\!\int_{\mathbb{R}^{2}}\{F(x\wedge y)-F(x)F(y)\}d(a_{1}-a_{2})(x)d(a_{1}-a_{2})(y) \\
&=&\int \!\!\!\int_{\mathbb{R}^{2}}\{F(x\wedge y)-F(x)F(y)\}d a(x)da(y).
\end{eqnarray*}
Then (\ref{cor:A-B-Kernel}) holds for $a=a_1-a_2$ and $b=b_1-b_2$ by polarization:
\begin{equation*}
Cov(a(X), b(X))
=\frac{1}{4}\left\{ Var(a(X)+ b(X))-Var(a(X)- b(X))\right\} .
\end{equation*}
\end{proof}

\begin{rem}
If $a$ is non-decreasing and right-continuous, then the identity in Corollary~\ref{cor:Shorack-Conseq}
can fail: for example, if 
$F(x)=(1-p)1_{[0,\infty)}(x)+p1_{[1,\infty )}(x)$ so that 
$X\sim \ $Bernoulli$(p)$, and 
$a(x)=1_{[1,\infty )}(x)$, then $a(X)\sim \ $Bernoulli$(p)$ so 
$Var(a(X))=p(1-p)$ on the left side, but 
\begin{equation*}
\int \!\!\!\int \{F(x\wedge y)-F(x)F(y)\}da(x)da(y)=(F(1)-F(1)F(1))\cdot 1\cdot 1=0.
\end{equation*}
(On the other hand, if $a(x)=1_{(0,\infty )}(x)$, then $a(X)=1$ with
probability $p$ so that it is again a Bernoulli$(p)$ random variable and the
left side is again $p(1-p)$, but the right side equals 
$F(0)-F(0)^{2}=(1-p)-(1-p)^{2}=(1-p)\cdot p$.) \smallskip 
\end{rem}

\begin{rem}
Note that both sides in the variance 
version of (\ref{cor:A-B-Kernel})  are infinite  if $Var(a(X))=\infty $. \smallskip
\end{rem}

\begin{proof}[Proof of Corollary~\protect\ref{cor:IndicatorAFcn}]
Let $a(x)\equiv F(z)-1_{[x\leq z]}$. First notice that
\begin{equation*}
F(z)\int_{\mathbb{R}}b(y)dF(y)-\int_{(-\infty, z]}b\left( y\right) dF\left( y\right) 
=\int_{\mathbb{R}}b\left( y\right) a\left( y\right) dF\left( y\right) \text{ .}
\end{equation*}
Then $a$ increases from $F(z)-1$ to $F(z)$ with the only change being a jump
upward of $1$ at $x=z$. Note that the first equality in (\ref{Indicator-MRLCovIdent-1}) %CovIdentity3-A}) 
holds since $E_{F}a(X)=0$. Then the
second equality in (\ref{Indicator-MRLCovIdent-1}) 
%CovIdentity3-A}) 
follows from (\ref{cor:A-B-Kernel}).
The  equalities in (\ref{Indicator-MRLCovIdent-2}) follow by noting that 
$$
a(x) = - \{ (1-F(z)) - (1- 1_{[x\le z]} ) \} = -  (1-F(z)) +  1_{[x> z]} .
$$
If $b \in L_{1}(F)$ is absolutely continuous,
then
\begin{equation*}
\int_{\mathbb{R}}K_{\mu }\left( z,y\right) db(y)
=\int_{\mathbb{R}}K_{\mu}\left( z,y\right) b^{\prime }(y)dy\text{ .}
\end{equation*}
Moreover, in this case, $b$ has bounded variation, so $b=b_{1}-b_{2}$
with $b_{j}$ non-decreasing, $j=1,2$,  Since $b$ is continuous and in $L_{1}\left( F\right) $, 
we may assume without loss of generality that $b_{j}$, $j=1,2$, 
are also continuous and in $L_{1}\left( F\right) $ 
(see for instance \cite{MR1762415}, Exercise 4.1, p.75). 
Hence, (\ref{Indicator-MRLCovIdent-1}) is %(\ref{rep_dg_L1}) is
valid in this case and so is (\ref{Indicator-MRLCovIdent-1-AbsC}). %(\ref{rep_g_prime_L1-MRLa}).  
The proof of (\ref{Indicator-MRLCovIdent-2-AbsC}) %(\ref{rep_g_prime_L1-MRLb}) 
is similar using (\ref{Indicator-MRLCovIdent-2}).  %(\ref{CovIdentity3-B}).
\end{proof}

%\begin{proof}[Proof of Corollary~\protect\ref{cor:A-B-Kernel}]

%\end{proof}
\medskip

\begin{proof}[Proof of Corollary~\protect\ref{cor:DifferenceVersion1}]  
%In Covariance identity 3 we see that $g=1_{[\cdot \leq z]}$ is bounded and
%If  $L_{\infty }(F)$, and this gives the opportunity to relax the
%assumption that $h\in L_2(F)$. 
Since $a \in L_p (F)$ and $b\in L_{q}(F)$ for some $p,q \in [1,\infty]$ satisfying
 $p^{-1}+q^{-1}=1$, then
by H\"{o}lder's inequality 
\begin{equation*}
|Cov(a(X),b(X))|\leq \Vert a(X)\Vert _{p}\Vert b(X)\Vert _{q}<\infty 
\end{equation*}
where 
\begin{equation*}
\Vert a(X)\Vert _{p}\equiv \left\{ E|a(X)|^{p}\right\} ^{1/p}
=\left\{ \int |a(x)|^{p}dF(x)\right\} ^{1/p}
\end{equation*}
and similarly for $\Vert b(X)\Vert _{q}$. 
The fact that $a$ and $b$ are absolutely
continuous with $a\in L_{p}(F)$, $b\in L_{q}(F)$ implies that $a=a_{1}-a_{2}$
and $b=b_{1}-b_{2}$ where $a_{j}$, $b_{j}$ are non-decreasing, left -
continuous and 
$\Vert a_{j}(X)\Vert _{p}<\infty $, 
$\Vert b_{j}(X)\Vert_{q}<\infty $ for $j=1,2$. 
Hence, Corollary~\ref{cor:DifferenceVersion1}  %{cor:ExtendedCovIdenityFour} %4 
follows from Corollary~\ref{cor:IndicatorAFcn}.  %Covariance identity 4. \
\end{proof}

%\begin{proof}[Proof of Corollary~\protect\ref{cor:KernelRecoveryOfDensity}] 
%\end{proof}

%In Covariance identity 3 we see that $g=1_{[\cdot \leq z]}$ is bounded and
%hence in $L_{\infty }(F)$, and this gives the opportunity to relax the
%assumption that $h\in L_2(F)$. If we assume that $h\in L_{q}(F)$ for some 
%$q\in \lbrack 1,\infty ]$ and $g\in L_{p}(F)$ where $p^{-1}+q^{-1}=1$, then
%by H\"{o}lder's inequality 
%\begin{equation*}
%|Cov(g(X),h(X))|\leq \Vert g(X)\Vert _{p}\Vert h(X)\Vert _{q}<\infty 
%\end{equation*}
%where 
%\begin{equation*}
%\Vert g(X)\Vert _{p}\equiv \left\{ E|g(X)|^{p}\right\} ^{1/p}
%=\left\{ \int |g(x)|^{p}dF(x)\right\} ^{1/p}
%\end{equation*}
%and similarly for $\Vert h(X)\Vert _{q}$. This leads to a further covariance
%identity as follows: 
%\smallskip 
%\noindent \textbf{Proof.} 
%\hfill $\Box$ 
%\smallskip 

\begin{proof}[Proof of Corollary~\protect\ref{cor:KernelRecoveryOfDensity}]
%\noindent \textbf{Proof.} 
(a) This follows from Corollary~\ref{cor:IndicatorAFcn}. %cor:ExtendedCovEqualityOne} (the extended version of Covariance identity 3). 
Indeed, we take $b=\varphi ^{\prime }$
in Corollary~\ref{cor:IndicatorAFcn}. Then from Corollary~\ref{cor:IndicatorAFcn}, %cor:ExtendedCovEqualityOne}, 
we get 
\begin{eqnarray*}
\int_{\mathbb{R}}K_{\mu }(x,y)d\varphi ^{\prime }(y) 
&=&-\int_{\mathbb{R} } \left( 1_{(-\infty ,x]}(y)-F(x)\right) \varphi ^{\prime }(y)f(y)dy \\
&=&-\int_{(-\infty, x]}\varphi ^{\prime }(y)f(y)dy+F(x)\int_{-\infty }^{\infty }\varphi ^{\prime }(y)f(y)dy \\
&=&f(x)+F(x)\cdot 0=f(x)
\end{eqnarray*}
since $\int_{(-\infty,x]}\varphi ^{\prime }(y)f(y)dy=-f(x)$ and $\int_{(x,\infty) }\varphi ^{\prime }(y)f(y)dy=f(x)$. 
%\textcolor{cyan}{(Adrien: Do the two latter identities hold without the assumption that $\varphi$ is absolutely continuous ?)}

(b):  This follows from (a) and the hypothesized absolute continuity.

(c) and (d): It remains only to show that the hypotheses of Corollary~\ref{cor:KernelRecoveryOfDensity}(a)
always hold in the log-concave case. But in this case $\varphi ^{\prime }$
is monotone non-decreasing, so by taking the right-continuous version of 
$\varphi ^{\prime }$ and letting 
$x_{0}\equiv \inf \{y\in \mathbb{R}:\varphi^{\prime }(y)\geq 0\}$, 
we have $\varphi ^{\prime }(x)\geq 0$ for $x\geq x_{0}$ and $\varphi ^{\prime }(x)<0$ for $x<x_{0}$. It follows that 
\begin{eqnarray*}
\int_{\mathbb{R}}|\varphi ^{\prime }(y)|e^{-\varphi (y)}dy
&=&\int_{x_{0}}^{\infty }\varphi ^{\prime }(y)e^{-\varphi
(y)}dy+\int_{-\infty }^{x_{0}}(-\varphi ^{\prime }(y))e^{-\varphi (y)}dy \\
&=&2e^{-\varphi (x_{0})}<\infty .
\end{eqnarray*}
%\textcolor{magenta}{Jon:  changed.  Please check.}\\
Thus $\varphi ^{\prime }\in L_{1}(F)$ and the hypotheses of Corollary~\ref{cor:KernelRecoveryOfDensity}(a)
hold.
\end{proof}

\begin{rem}
Another way to check finiteness of covariances is to use the following 
consequence of H\"older's inequality for the kernel $K_\mu$ in (\ref{def_kernel}):
note that if $1/p+1/q=1$ with $p\geq 1$ we have 
\begin{eqnarray*}
\lefteqn{F(x\wedge y)-F(x)F(y)   =Cov(1_{[X\leq x]},1_{[X\leq y]})} \\
&=&\int \left\{ (1_{[z\leq x]}-F(x))(1_{[z\leq y]}-F(y))\right\} dF(z) \\
&\leq &\left\{ \int |1_{[z\leq x]}-F(x)|^{p}dF(z)\right\} ^{1/p}\cdot
\left\{ \int |1_{[z\leq y]}-F(y)|^{q}dF(z)\right\} ^{1/q} \\
&=&\left\{ |1-F(x)|^{p}F(x)+F(x)^{p}(1-F(x))\right\} ^{1/p}\cdot \left\{
|1-F(y)|^{q}F(y)+F(y)^{q}(1-F(y))\right\} ^{1/q} \\
&\leq &\left\{ F(x)(1-F(x))\right\} ^{1/p}\cdot \left\{ F(y)(1-F(y))\right\}
^{1/q}.
\end{eqnarray*}%
Thus finiteness of the integrals $\int_{\mathbb{R}}\left\{
F(x)(1-F(x))\right\} ^{1/p}d a_{j}(x)$ and \newline
$\int_{\mathbb{R}}\left\{ F(x)(1-F(x))\right\} ^{1/q}db_{j}(x)$ for $j=1,2$
and any conjugate pair $(p,q)$ implies that $|Cov(a(X),b(X))|<\infty $
follows from Corollary~\ref{cor:Shorack-Conseq}.
\end{rem}

\subsection{Examples and Counterexamples}

We give five examples in connection with the formulas (\ref{LogConcaveTypeIdentityGeneralForm}) 
and (\ref{LogConcaveTypeIdentitySecondDerivativeForm}) 
in Corollary~\ref{cor:KernelRecoveryOfDensity}.    
In the first three examples $f$ is log-concave, (\ref{LogConcaveTypeIdentityGeneralForm}) follows 
from (c) and is known from \cite{Menz-Otto:2013}.  % Menz and Otto
The third and fourth examples give cases in which log-concavity fails, but at least
one of (\ref{LogConcaveTypeIdentityGeneralForm}) and (\ref{LogConcaveTypeIdentitySecondDerivativeForm}) holds.
\smallskip

%begin{exmpl}
%\label{ex:Gamma}
\noindent 
\textbf{Example 1.} 
(Gamma densities). Let 
\begin{equation*}
f(x) \equiv f_{\theta }(x)=\frac{x^{\theta -1}}{\Gamma (\theta )}\exp (-x)1_{(0,\infty)}(x)
\end{equation*}
for $\theta >0$. It is easily seen that the densities $f_{\theta }$ are
log-concave for $\theta \ge 1$ and absolutely continuous for $\theta>1$. 
The derivative $\varphi^{\prime} $ exists everywhere if $\theta >1$.
 Thus (\ref{LogConcaveTypeIdentityGeneralForm}) holds for $\theta>1$.
%Do the hypotheses of Corollary 3 hold for all $%
Furthermore, note that 
\begin{equation*}
\varphi (x)=x-(\theta -1)\log x+\log \Gamma (\theta )
\end{equation*}
and hence $\varphi^{\prime} $ is absolutely continuous for $\theta >1$ with 
\begin{equation*}
\varphi ^{\prime }(x)=1-\frac{\theta -1}{x},
\qquad \varphi ^{\prime \prime}(x)=\frac{\theta -1}{x^{2}}
\end{equation*}
while 
$$
E| \varphi^{\prime} (X) | \le (\theta-1) \int_0^{\infty} x^{-1} \frac{x^{\theta -1}}{\Gamma (\theta )}\exp (-x) dx = 1 < \infty
$$
if $\theta > 1$.    
Thus by (b) of Corollary~\ref{cor:KernelRecoveryOfDensity},
 (\ref{LogConcaveTypeIdentitySecondDerivativeForm})   holds for $\theta>1$.    
 When $\theta =1$, $f(x) = \exp(-x) 1_{(0,\infty)} (x)$ 
 is log-concave, but $\varphi $ is not absolutely continuous, and it can easily be seen that 
 (\ref{LogConcaveTypeIdentityGeneralForm}) fails.  When $\theta \in (0,1)$, $f_{\theta}$ is not log-concave and $\varphi$ is 
 not absolutely continuous.  In this case 
 the hypotheses (and conclusions) of Corollary~\ref{cor:KernelRecoveryOfDensity} fail.
 \medskip
 %\end{exmpl}
%\textcolor{magenta}{Jon:  what happens when $\theta\le 1$?  $\theta =1$ is the ``possible counterexample'' given 
%just after the corollary statement.  It seem likely that the identity fails for $0<\theta<1$ as well.}
%Thus $\varphi _{1}^{\prime }(x)=1-(\theta -1)/x$ and 
%$\varphi _{2}^{\prime}(x)=0$ satisfy the hypotheses of Corollary 3 since 
%$x^{-1}\in L_{1}(F)$ if $ \theta >1$ and $\varphi ^{\prime }(x)=2-\theta +\int_{1}^{x}\varphi ^{\prime\prime }(y)dy$. 
%Thus the identity (\ref{LogConcaveTypeIdentityGeneralForm}) 
%LogConcaveTypeIdentity}) 
%holds for all $\theta >1$. 
%\medskip

\noindent 
\textbf{Example 2.} 
%\begin{exmpl}
%\label{ex:logistic}
(Logistic density).   Now let 
$f$ be the logistic density,  $f(x) = e^{-x}/(1+e^{-x})^2$.  In this case $f$ is absolutely continuous and strictly log-concave
since $\varphi (x) = x + 2 \log (1+e^{-x})$ is convex with $\varphi^{\prime \prime} (x) = 2f(x) >0$ for all $x \in \RR$ and 
$\varphi^{\prime}$ is bounded.  
Thus  (\ref{LogConcaveTypeIdentitySecondDerivativeForm}) holds.  
This can also be verified by a direct calculation:
\begin{eqnarray*}
\int_{\RR} K_{\mu} (x,y) \varphi^{\prime\prime} (y) dy 
& = & \int_{(-\infty,x]} F(y) (1-F(x)) 2 f(y) dy + \int_{(x,\infty)} F(x) (1-F(y) ) 2 f(y) dy \\
& = &  2(1-F(x)) \int_{(-\infty,x]} F(y) d F(y)  + 2F(x) \int_{(x,\infty)} (1-F(y) )dF(y) \\
& = & F(x)(1-F(x)) \{ F(x) + (1-F(x)) \} = F(x) (1-F(x)) = f(x) .
\end{eqnarray*}
%\end{exmpl}
%\medskip

\noindent 
\textbf{Example 3.} 
%\begin{exmpl}
%\label{ex:laplace}
(Laplace density). 
For the Laplace density $f(x)=(1/2)\exp (-|x|)$, we see that $f$ is absolutely continuous,
$\varphi (x)=|x|+\log (2)$ is convex, and $\varphi $ has derivative $\varphi' (x) = \mbox{sign}(x)$ for $x \neq 0$.  
In this case $f$ is log-concave and (\ref{LogConcaveTypeIdentityGeneralForm}) holds by (c) 
of Corollary~\ref{cor:KernelRecoveryOfDensity}.
%\end{exmpl}
\medskip
%so $\varphi^{\prime }(x)=\mbox{sign}(x)
%=1_{(0,\infty )}(x)-1_{(-\infty ,0]}(x)=\varphi_{1}^{\prime }(x)-\varphi _{2}^{\prime }(x)$ 
%where $\varphi _{j}^{\prime }$
%satisfy the hypotheses of Covariance identity 3. A computation similar to
%that of Corollary 3 yields 
%\begin{equation*}
%\int_{\mathbb{R}}K_{\mu }(x,y)\,d\varphi ^{\prime }(y)=p(x).
%\end{equation*}%
%This can also be checked by direct computation in this particular case.

\par\noindent
%noindent 
\textbf{Example 4.} 
%\begin{exmpl}
%\label{ex:cauchy} 
(Cauchy). Suppose that $f$ is the Cauchy
density given by 
\begin{equation*}
f(x)=\frac{1}{\pi }\frac{1}{1+x^{2}}.
\end{equation*}
Then 
\begin{equation*}
F(x)=\frac{1}{2}+\frac{1}{\pi }\arctan (x),
\end{equation*}
$f$ is absolutely continuous,
$\varphi (x)=\log \pi +\log (1+x^{2})$, and $\varphi^{\prime}$ is absolutely continuous with 
\begin{equation*}
\varphi ^{\prime }(x)=\frac{2x}{1+x^{2}},
\ \ \ \varphi ^{\prime \prime }(x)=\frac{2(1-x^{2})}{(1+x^{2})^{2}}.
\end{equation*}
It follows from Corollary~\ref{cor:KernelRecoveryOfDensity} (b) that 
the identity (\ref{LogConcaveTypeIdentitySecondDerivativeForm}) holds.
This can also be seen by direct calculation as follows: 
\begin{eqnarray*}
\int_{\mathbb{R}}K_{\mu }(x,y)\varphi ^{\prime \prime }(y)dy
&=&(1-F(x))\int_{(-\infty ,x]}F(y)\frac{2(1-y^{2})}{(1+y^{2})^{2}}dy 
         \ +\ F(x)\int_{(x,\infty) }(1-F(y))\frac{2(1-y^{2})}{1+y^{2})^{2}}dy \\
&=&\frac{2(1-F(x))(1+\pi x+2x\arctan (x))+2F(x)(1-\pi x+2x\arctan (x)}{2\pi (1+x^{2})} \\
&=&\frac{(1-F(x))\left( 1+\pi x+2x\arctan (x)\right) +F(x)\left( 1-\pi
        x+2x\arctan (x)\right) }{\pi (1+x^{2})} \\
&=&\frac{1+2\pi x\{F(x)(1-F(x))-F(x)(1-F(x))\}}{\pi (1+x^{2})} \\
&=&\frac{1}{\pi (1+x^{2})}=f(x).
\end{eqnarray*}
%\end{exmpl}

\noindent 
\textbf{Example 5.} 
%\begin{exmpl}
%\label{ex:bridge}
(Bridge distribution; \cite{MR2024756}).
Suppose that $X \sim f \equiv f_{\theta} $ where, for $\theta \in (0,1)$, 
\begin{equation*}
f_{\theta} (x) = \frac{\sin( \pi \theta)}{2 \pi ( \cosh (\theta x) + \cos (
\pi \theta ))}.
\end{equation*}
These densities are log-concave for $\theta \in (0,1/2]$, but log-concavity
fails for $\theta \in (1/2,1)$.   They are all absolutely continuous.   
Here with $\varphi_{\theta} (x) \equiv -\log p_{\theta} (x)$ we have 
\begin{eqnarray*}
\varphi_{\theta}^{\prime}(x) = \frac{\theta \sinh (\theta x)}{ \cos (\pi
\theta) + \cosh (\theta x)}
\end{eqnarray*}
and 
\begin{eqnarray*}
\varphi_{\theta}^{\prime\prime}(x) = \frac{\theta^2 (1+ \cos(\pi \theta )
\cosh (\theta x)}{( \cos (\pi \theta) + \cosh(\theta x))^2} .
\end{eqnarray*}
Note that $\varphi_{\theta}^{\prime} $ is bounded, and hence $\varphi_{\theta}^{\prime} \in L_1 (F_{\theta})$.
It follows from  Corollary~\ref{cor:KernelRecoveryOfDensity} (b) 
that (\ref{LogConcaveTypeIdentitySecondDerivativeForm}) holds.
%\end{exmpl}
%With the help of 
%\cite{MR0362689}, A2, page 64, it is easily seen that the hypotheses of Lemma 2 
%hold for all $\theta \in (0,1)$, and hence the conclusions hold as well.\\
%\textcolor{magenta}{Jon:  what is Lemma 2?  Some argument still needed?}
%\medskip
%\medskip

\section*{Acknowledgements:}  We owe thanks to two referees for catching a number of 
errors and for helpful suggestions concerning the organization and exposition.

\bibliographystyle{imsart-nameyear}
\bibliography{chern}

\end{document}